\documentclass[twoside,reqno,12pt] {amsart}
\usepackage[left=1in,right=1in,top=1in,bottom=1in]{geometry}

\usepackage[hidelinks]{hyperref}
\usepackage{amsfonts,amsmath,amsthm}
\usepackage{mathtools}
\usepackage{amssymb}
\usepackage{amsaddr}
\usepackage{physics}
\usepackage{color}
\usepackage{graphics}
\usepackage{comment}
\usepackage{stmaryrd}

\newcommand{\HOX}[1]{\marginpar{\footnotesize sharp1}}

\newtheorem{thm}{Theorem}[section]
\newtheorem{cor}[thm]{Corollary}
\newtheorem{lem}[thm]{Lemma}
\newtheorem{prop}[thm]{Proposition}

\newtheorem{defn}[thm]{Definition}
\newtheorem{rem}[thm]{Remark}
\newtheorem{invprob}[thm]{Inverse Problem}

\theoremstyle{definition}

\numberwithin{equation}{section}

\newcommand{\thmref}[1]{Theorem~\ref{sharp1}}
\newcommand{\secref}[1]{\S\ref{sharp1}}
\newcommand{\lemref}[1]{Lemma~\ref{sharp1}}

\newcommand{\C}{\mathbb{C}}

\newcommand{\N}{\mathbb{N}}

\newcommand{\R}{\mathbb{R}}
\newcommand{\cR}{\mathcal{R}}
\newcommand{\cS}{\mathcal{S}}

\renewcommand{\t}[1]{\tilde{sharp1}}

\newcommand{\dbtilde}[1]{\tilde{\tilde{sharp1}}}

\def\tilde{\widetilde}
\def \bfo {\begin {eqnarray*} }
\def \efo {\end {eqnarray*} }
\def \ba {\begin {eqnarray*} }
\def \ea {\end {eqnarray*} }
\def \beq {\begin {eqnarray}}
\def \eeq {\end {eqnarray}}
\def \supp {\hbox{supp }}
\def \diam {\hbox{diam }}

\def \p {\partial}

\def \cL {\mathcal{L}}

\def\V{\mathcal V}
\def\L{\mathcal{L}_{A,q}}

\DeclarePairedDelimiter{\inner}{\langle}{\rangle}

\def\tilde{\widetilde}
\def \bfo {\begin {eqnarray*} }
\def \efo {\end {eqnarray*} }
\def \ba {\begin {eqnarray*} }
\def \ea {\end {eqnarray*} }
\def \beq {\begin {eqnarray}}
\def \eeq {\end {eqnarray}}
\def \supp {\hbox{supp }}
\def \diam {\hbox{diam }}

\def \p {\partial}
\def \Box{\square_{A,q}}

\usepackage{color,comment}

\title[A Hyperbolic Inverse Problem on a closed manifold with disjoint data]
{A Hyperbolic Inverse Problem for lower order terms on a closed manifold with disjoint data
}

\subjclass[2020]{35R30, 35L05, 58J45, 86A22}

\author[Lassas]{Matti Lassas}
\address{Department of Mathematics and Statistics\\
University of Helsinki, 00014, Finland}
\email{matti.lassas@helsinki.fi}

\author[Liu]{Boya Liu}
\address{Department of Mathematics\\
North Dakota State University, Fargo\\ 
ND 58102, USA}
\email{boya.liu@ndsu.edu}

\author[Saksala]{Teemu Saksala}
\address{Department of Mathematics\\
North Carolina State University, Raleigh\\ 
NC 27695, USA}
\email{tssaksal@ncsu.edu}

\author[Shedlock]{Andrew Shedlock}
\address{Department of Mathematics\\
North Carolina State University, Raleigh\\ 
NC 27695, USA}
\email{ajshedlo@alumni.ncsu.edu}

\author[Zhao]{Ziyao Zhao}
\address{Department of Mathematics and Statistics\\
University of Helsinki, 00014, Finland}
\email{ziyao.zhao@helsinki.fi}

\begin{document}

\begin{abstract}
We study the unique recovery of time-independent lower order terms appearing in the symmetric first order perturbation of the  Riemannian wave equation by sending and measuring waves in disjoint open sets of \textit{a priori} known closed Riemannian manifold. In particular, we show that if the set where we capture the waves satisfies a geometric control condition as well as a certain local symmetry condition for the distance functions, then the aforementioned measurement is sufficient to recover the lower order terms up to the natural gauge.
For instance, our result holds if the complement of the receiver set is contained in a simple Riemannian manifold.
\end{abstract}

\maketitle

\section{Introduction and Statement of Results}
\label{sec:intro}

We study an inverse problem for a hyperbolic Cauchy problem on a closed Riemannian manifold. 
In particular, under two geometric assumptions, we establish a global unique recovery of time-independent lower order terms appearing in the symmetric first order perturbation of the wave equation by sending and measuring waves in disjoint open sets of an \textit{a priori} known closed Riemannian manifold. The data is formalized by a local source-to-solution map. This geometric inverse problem  is motivated by seismic exploration, which uses controlled sound waves to image and map the Earth's subsurface, revealing layers of rock and helping to locate potential deposits of oil and gas, as well as to understand groundwater, geological structures, and earthquake risks. 

Our proof strategy is built on the celebrated Boundary Control method (BC method), that was put forward by Belishev \cite{belishev_russian, belishev1990wave} and was further developed by its author and others \cite{avdonin1992boundary, belishev1989application, belishev1987nonstationary, belishev1989inverse} by combining finite-speed arguments with control theory to solve inverse problems for the wave equation in Euclidean domains. A turning point in the study of inverse problems for wave equation happened in 1995 when Tataru proved a Holmgren-type uniqueness theorem for wave equations with nonanalytic coefficients \cite{tataru1995unique, tataru1999unique}. This and the earlier results on the boundary control method by Belishev and Kurylev in \cite{belishev1992reconstruction} solved the inverse problem for the wave equation on a Riemannian manifold with complete boundary data. 
In the current paper we rely on an adaptation of the BC-method for the case of a closed manifold and the source-to-solution map that first appeared in \cite{source-to-solution}.

We assume that the set where we capture the waves satisfies a geometric control condition and a certain local symmetry condition for the distance functions.
Indeed, we attack this inverse problem with an approximate controllability argument, which is based on Tataru's unique continuation principle given in \cite{tataru1995unique}, from the set where we send the waves, 
and an application of exact controllability of the hyperbolic Cauchy problem from the set where we capture the waves. The former tool is applicable whenever the measurement time is long enough, whereas the latter is applicable in our case due to the geometric assumptions. 
The inverse problem considered in this paper is more difficult than the one solved in \cite{saksala2025inverse}, where the waves are sent and received in the same set, in which case we have full control on the norms of the waves. However, this control is lost in the current paper, which concerns  disjoint source and receiver sets. To medicate this issue, we need to impose the aforementioned geometric assumptions. 

In addition to the previously mentioned geometric assumptions, which are independent of the lower order terms that we aim to determine, we require that the spatial part of the hyperbolic operator is given by the magnetic Schr\"odinger operator. The imposed symmetry on the PDE is crucial when we prove Blagoveščenskii's identity, which allows us to calculate the $L^2$-inner products of any two waves that are sent from the source and receiver sets, respectively. Therefore, the limitations of the current work lie in (1) geometric requirements for the receiver set, (2) time-independence of lower order terms (the unique continuation principle for hyperbolic equations is known to fail if  coefficients are time-dependent \cite{Alinhac, Alinhac_Baoendi}), and (3) the specific structure of the partial differential operator. The relaxation of any of these three conditions will merit a paper of its own, and is therefore left outside the scope of this work.

\subsection{Problem setting and the main result}
Through out this paper, we let $(N,g)$ stand for a known closed Riemannian manifold, while the notations $A$ and $q$ are reserved for an unknown smooth real-valued co-vector field and an unknown smooth real-valued function, respectively. We denote by
\begin{align}
\label{eq:mag_schro_op}
\L:
=
(d+iA)^\ast(d+iA)+q
=
-\Delta_gu -2i\langle A,du\rangle_g + (id^*A + |A|_g^2 + q)u,
\end{align}
the elliptic self-adjoint first order perturbation of the Laplace-Beltrami operator $\Delta_g$, which is commonly known as the magnetic Shr\"odinger operator. Here $d$ is the exterior derivative with $d^\ast$ being its formal $L^2$-adjoint, $i$ is the imaginary unit, while $\langle \cdot, \cdot\rangle_g$ and $|\cdot|_g$ denote the Riemannian inner product and norm for co-vector fields, respectively.

For each measurement time $T>0$, we define the hyperbolic analog of $\L$ in the spacetime $(0,T)\times N$ as
\begin{align}
\label{eq:hyperbolic_operator} 
\Box
&:= 
\p_t^2+\L.
\end{align}
Then for each $f \in L^2((0,T) \times N)$, we consider the following hyperbolic Cauchy problem 
\begin{align}
\label{eq:Cauchy_Problem}
\begin{cases}
\Box u^f(t,x)=f(t,x) \quad \text{ in } (0,T)\times N,
\\
u(0,\cdot)=\p_tu(0,\cdot)=0 \quad \text{ in } N.
\end{cases}
\end{align}
We refer readers to \cite[Chapter 6]{taylor_partial} for the  wellposedness of this problem. 

The goal of this paper is to show that we can recover the unknown time-independent lower order terms $A$ and $q$ appearing in $\Box$ by sending waves from some open subset $\cS$ of $N$ and capturing them in another subset $\cR$. 
In particular, we define the local source-to-solution map of the problem \eqref{eq:Cauchy_Problem} associated to the open sets $\cS, \cR \subset N$ as
\begin{equation}
\label{eq:S_to_S_map}
\Lambda_{\cS,\cR,T}^{A,q}\colon C^\infty_0((0,T)\times \cS) \to C^\infty((0,T)\times \cR),
\quad 
\Lambda_{\cS,\cR,T}^{A,q}(f)=u^f|_{(0,T)\times \cR},
\end{equation}
and provide an answer to the following inverse problem.
\begin{invprob}
\label{ip}
What are the sufficient assumptions for $\cS$, $\cR$, and $T>0$ such that the respective source-to-solution map $\Lambda_{\cS,\cR,T}^{A,q}$ determines the co-vector field $A$ and the function $q$ uniquely?
\end{invprob}

Inverse Problem \ref{ip} has a natural gauge due to the first order term.
Indeed, if $\kappa$ is a smooth complex-valued  unitary function on $N$ that equals 1 in the set $\cS \cup \cR$,  it follows from an argument similar to the one in \cite[Proposition 2.31]{saksala2025inverse} that $\kappa\L \kappa^{-1}=\mathcal{L}_{A+i\kappa^{-1}d\kappa,q}$ is also a magnetic Schr\"odinger operator, and the associated source-to-solution map satisfies the equation 
\[
\Lambda^{A,q}_{\cS,\cR,T}(f)
= \Lambda^{A+i\kappa^{-1}d\kappa,q}_{\cS,\cR,T}(f) \quad \text{for all $f \in C_0^\infty((0,\infty)\times \cS)$}.
\]
Hence, the local source-to-solution map $\Lambda^{A,q}_{\cS,\cR,T}$ does not uniquely determine the co-vector field $A$. In this paper we prove the converse of this result under two geometric conditions that are both independent of $A$ and $q$.

In what follows we use the notations $TN$ and $SN$ for the tangent and unit sphere bundles of $(N,g)$, respectively. In particular, since $N$ is compact, each $(x,\xi) \in SN$ determines a unique unit speed geodesic $\gamma_{x,\xi}$ from $\R$ to $N$ given by the initial conditions $\gamma_{x,\xi}(0)=x, \: \dot\gamma_{x,\xi}(0)=\xi$. The notation $d(\cdot,\cdot)$ is reserved for the Riemannian distance function and should not be confused with the exterior derivative introduced previously. 

We now set the following two standing assumptions:

\noindent
\textbf{Assumption 1:}
\textbf{Geometric Control Condition from the set $\cR$}
\begin{align}
\label{eq:GCC}
\tag{A1}
\text{For each $(x,\xi)\in SN$, there exists $t \in [0,T]$ such that $\gamma_{x,\xi}(t) \in \cR$}. 
\end{align}

\noindent
\textbf{Assumption 2:} \textbf{Symmetry of distances with respect to the set $\cR$}
\begin{equation}
\label{eq:assumption_rev}
\tag{A2}
\begin{aligned}
&\text{Every point } x\in  N \text{ has an open neighborhood } U\subset N \text{ such that for all distinct }p,q\in U,
\\
&\text{there exist } y,z \in\cR 
\text{ such that $d(p,z)<d(q,z)$ and $d(q,y)<d(p,y)$}.    
\end{aligned}
\end{equation}

Our main theorem is as follows:

\begin{thm}
\label{thm:basic_uniqueness_thm}
Let $(N,g)$ be a smooth closed Riemannian manifold and $T>\diam(N,g)$. Let $\cR\subset N$ be an open set that satisfies the properties  \eqref{eq:GCC} and \eqref{eq:assumption_rev}.
If $\cS \subset N$ is any open set, and $(A_i,q_i)$, $i \in \{1,2\}$, are two pairs of a smooth co-vector field and a smooth function such that 
\begin{equation}
\label{eq:hypothesis}
\Lambda_{\cS,\cR,2T}^{A_1,q_1}=\Lambda_{\cS,\cR,2T}^{A_2,q_2},
\end{equation}
then $q_1=q_2$ and there exists a smooth unitary function $\kappa\colon N \to \C$ such that  $\kappa(z)=1$ for all $z \in \cR \cup \cS$ and 
\[
A_1=A_2+i\kappa^{-1}d\kappa.
\]
\end{thm}

\subsection{Novelty and scope of this paper}

If $\cR \cap \cS \neq \emptyset$, as recently established in \cite{saksala2025inverse} by the third and the fourth authors, the conclusion of Theorem \ref{thm:basic_uniqueness_thm} holds without any geometric assumptions. Thus, the main contribution of the current paper lies in the extension of this result to the case when the sets $\cR$ and $\cS$ are disjoint. In order to pursue such an extension, we need to impose the geometric assumptions \eqref{eq:GCC} and \eqref{eq:assumption_rev} on the receiver set $\cR$.  

For instance, the set $\cR$ satisfies the conditions \eqref{eq:GCC} and \eqref{eq:assumption_rev} whenever $(N,g)$ is a closed Riemannian manifold, and the open set $\cR\subset N$ is sufficiently  large that $N\setminus \cR$ is contained in a simple Riemannian manifold. 
Here a simple manifold means a compact Riemannian manifold with smooth strictly convex boundary such that every pair of points is connected by a unique smoothly varying distance minimizing geodesic. These properties imply that simple manifolds do not have any trapped geodesics, as all geodesics have a distance minimizing extension up to the boundary. Hence, none of them can be longer than the diameter of the manifold. The geometric assumption that the complement of the set $\cR$ is simple was  used very recently in \cite{feizmohammadi2025inverse}. 

There are several examples when the condition \eqref{eq:GCC} holds without requiring the  convexity of the complement of $\cR$. For instance, this occurs when $N$ is a sphere and $\cR$ contains a neighborhood of a half circle connecting two antipodal points, or   when $N$ is the flat torus obtained by gluing together the edges of a square and the set $\cR$ is a neighborhood of the edges. Moreover, as the following proposition indicates, the condition \eqref{eq:GCC} is stronger than the property \eqref{eq:assumption_rev} in some geometries:

\begin{prop}
\label{prop:example_GCC}
If an open subset $\cR \subset \mathbb{S}^2$ satisfies  \eqref{eq:GCC}, it also satisfies   \eqref{eq:assumption_rev}.
Moreover,  the conclusion of Theorem \ref{thm:basic_uniqueness_thm} holds in this case.
\end{prop}
\begin{proof}
Suppose, for the sake of contradiction,   that $\cR$ satisfies \eqref{eq:GCC} but not \eqref{eq:assumption_rev}.
Let $x \in \mathbb{S}^2$. Since $\cR$ does not satisfy  \eqref{eq:assumption_rev}, for each $r\in (0,\pi]$ there exist distinct points $p,q \in B(x,r) \subset \mathbb{S}^2$ such that $d(z,p)\leq d(z,q)$ for all $z \in \cR$.

After a rotation if necessary, we assume that in polar coordinates $\{(\theta,\varphi) :\theta\in [-\pi,\pi),\ \varphi\in [-\frac{\pi}{2},\frac{\pi}{2}]\}$, we have $p=(0,a)$ and $q=(0,-a)$ for some $0<a< r$. 
We then define the set  
\[
\Sigma:=\{z\in \mathbb{S}^2 \mid d(z,p)\leq d(z,q)\}, 
\]
and note that $\Sigma$ is the closed hemisphere containing $p$, whose equator $\gamma=\p \Sigma$ intersects the half circle $\{(0,\varphi): \: \varphi\in [-\frac{\pi}{2},\frac{\pi}{2}] \}$ orthogonally at $(0,0)$. 
Since $\cR\subset \Sigma$ and  $\cR$ is open, we also have   $\gamma \cap \cR =\emptyset$. Hence, the property \eqref{eq:GCC} does not hold either. This completes the proof of Proposition \ref{prop:example_GCC}.
\end{proof}
\noindent
In light of this result, it seems possible that Theorem \ref{thm:basic_uniqueness_thm} could be proved without explicitly assuming \eqref{eq:assumption_rev}. However, this generalization  is beyond the scope of the present paper.

Our main result can be seen as a closed manifold analog of \cite{kian2019unique}, where  the authors established the recovery of the lower order terms from the hyperbolic Dirichlet-to-Neumann map of the initial boundary value problem
\begin{align}
\label{eq:IBVP}
\begin{cases}
\Box v_h=0   & \text{ in } (0,T)\times M,
\\
v_h(0,\cdot)=\p_t v_h(0,\cdot)=0 & \text{ in } M,
\\
v_h=h & \text{ on } (0,T)\times \p M,
\end{cases}
\end{align}
where $(M,g)$ is a smooth compact Riemannian manifold with a smooth boundary $\p M$. In \cite{kian2019unique} the Dirichlet data is supported on a subset of the boundary that satisfies the condition \eqref{eq:GCC} for billiard geodesics, whereas the subset of the boundary where the Neumann data is collected needs to be strictly convex. As we have discussed above, the convexity assumption of \cite{kian2019unique} is stronger than our assumption \eqref{eq:assumption_rev}, and in our paper we provide a global determination of the lower order terms, while \cite{kian2019unique}  provides the determination of $A$ and $q$ only in the convex hull of the set where the Neumann data is measured. Moreover, we show that the gauge $\kappa\equiv 1$ in $\cR \cup \cS$, whereas in \cite{kian2019unique} this was only established on $\cR$.

Inverse Problem \ref{ip} considered in this paper has an interesting uncommon feature. It is known that the recovery of the leading order terms, namely the metric $g$, from the source-to-solution map, as in \eqref{eq:S_to_S_map}, can be accomplished under weaker assumptions than those required for the recovery of lower order terms presented in this paper. For instance, it was established in \cite{Lassas_Nul_Oksanen_Ylinen} that in the absence of the lower order terms, the source-to-solution map determines the Riemannian metric $g$ up to isometry if the set $\cR$ satisfies a weaker version of Assumption \eqref{eq:GCC}.
We shall discuss this condition in detail in Section \ref{sec:from_CGG_to_Exact_controllabity}, where we also show that  \eqref{eq:GCC} yields the control operator of the Cauchy problem \eqref{eq:Cauchy_Problem}
\begin{equation}
\label{eq:control_op}
\begin{aligned}
C : L^2((0,T)\times \cR)\to H^1(N)\times L^2(N), \quad 
C(f):=(-u^{{f}}(T,\cdot),\p_t u^{{f}}(T,\cdot)),
\end{aligned}
\end{equation}
where $\supp (f)\subset [0,T] \times \overline \cR$, being onto. This property is called the exact controllability of the problem \eqref{eq:Cauchy_Problem} from the set $\cR$. The initial boundary value problem \eqref{eq:IBVP} has an analogous control operator
\[
\begin{aligned}
\tilde C : L^2((0,T)\times  \Gamma)\to L^2(N)\times H^{-1}(N), \quad 
\tilde C(h):=(-v_{h}(T,\cdot),\p_t v_{h}(T,\cdot)),
\end{aligned}
\]
whenever $\Gamma \subset \p M$ is open and $\supp(h) \subset [0,T]\times \overline \Gamma$. The key difference between these two control operators is that $C$ is smoothing in the spatial variables but $\tilde C$ is not. In particular, the  exact controllability of  $\tilde C$ guarantees an $L^2$-to-$L^2$-onto property for the waves, while for $C$ the $L^2$-to-$L^2$-onto property only holds for the time derivatives of the waves. In our approach to solve Inverse Problem \ref{ip}, we need to take this into account and prove several key identities for the waves and their time derivatives. This complication was not present in \cite{kian2019unique} due to the mapping properties of the operator $\tilde C$.

\subsection{Previous literature}
In inverse problems of determining coefficients of PDEs,
it is common to use
the Dirichlet-to-Neumann (DN) map or the Cauchy data to model measurements.
Aside  from their applications in seismic explorations and medical ultrasound imaging, this choice is reasonable since several other types of data can be reduced to or are equivalent to the DN map \cite{KKL,Katchalov_data_equivalence}.
Meanwhile, in this paper we shall adapt the techniques originally developed for the DN map case to the case of local source-to-solution maps instead of reducing Inverse Problem \ref{ip} to the DN map. 

Let us next turn our attention to inverse problems with disjoint partial data, which are significant in many fields such as medical imaging, geology, physics, and engineering. In this challenging measurement setting all the results that we are aware of rely on a version of the exact controllability of the hyperbolic equation. The works that we survey next are based research direction that was first initiated in  
\cite{belishev1997uniqueness, kurylev1997multidimensional} and studied the recovery of the coefficients appearing in a hyperbolic partial differential operator from the full DN map under under an exact controllability assumption.
The authors of \cite{lassas2014inverse} showed that the DN map for the wave equation, where the Dirichlet and Neumann data are measured on some disjoint relatively open subsets of the boundary, uniquely determines a compact Riemannian manifold with boundary up to an isometry. This result was obtained under the Hassell–Tao condition for eigenvalues and eigenfunctions of the Dirichlet Laplacian on the subset of the boundary where the Neumann data is measured. This condition is valid for instance if the respective wave equation is exactly controllable. 
It was subsequently established in \cite{kian2019unique} that the disjoint partial DN map determines the first order perturbation of the wave operator uniquely up to the natural gauge invariance in a neighborhood of the set where the Neumann data is measured. This result was obtained by assuming that this set is strictly convex, and that the wave equation is exactly controllable from the set where the Dirichlet data is supported. In addition, this work also contains a global result under a convex foliation condition.
We refer readers to \cite{Imanuvilov_Uhlmann_Yamamoto,Rakesh_Horn,Rakesh_Sacks} for several more disjoint partial data results.

We now survey some known literature related to inverse problems involving the source-to-solution map. When   sources and solutions are measured on the same open subset of a closed manifold, the authors of \cite{source-to-solution} showed that the source-to-solution map for the wave operator determines the Riemannian metric up to  isometry. A similar result for the fractional diffusion equation was obtained in \cite{Helin_Lassas_Ylinen_Zhang}. Furthermore, a very recent result \cite{saksala2025inverse} by the third and fourth authors extended \cite{source-to-solution} and showed that the local source-to-solution map of the hyperbolic operator \eqref{eq:hyperbolic_operator} determines a complete Riemannian manifold, as well as time-independent lower order terms, up to their natural obstructions.  

To the best of our knowledge, the only work addressing the case where the source and solution are measured on disjoint subsets of a closed manifold is \cite{Lassas_Nul_Oksanen_Ylinen}, where the authors established an analogue of the result in \cite{lassas2014inverse}, i.e., the determination of the Riemannian manifold. Their analysis was carried out in the absence of lower-order terms and required a spectral bound condition for the Laplace spectrum that is analogous to the Hassel-Tao condition considered in \cite{lassas2014inverse}.

Our proof strategy is built on the celebrated Boundary Control method (BC method), that was put forward by Belishev \cite{belishev_russian, belishev1990wave} and was further developed by its author and others \cite{avdonin1992boundary, belishev1989application, belishev1987nonstationary, belishev1989inverse} by combining finite-speed arguments with control theory to solve inverse problems for the wave equation in Euclidean domains. A turning point in the study of inverse problems for wave equation happened in 1995 when Tataru proved a Holmgren-type uniqueness theorem for wave equations with nonanalytic coefficients \cite{tataru1995unique, tataru1999unique}. This and the earlier results on the boundary control method by Belishev and Kurylev in \cite{belishev1992reconstruction} solved the inverse problem for the wave equation on a Riemannian manifold with complete boundary data. 
In the current paper we rely on an adaptation of the BC-method for the case of a closed manifold and the source-to-solution map that first appeared in \cite{source-to-solution}.
We refer readers to \cite{KKL} for a thorough review of related literature.

Due to their dependency on the unique continuation principle presented in \cite{tataru1995unique}, our methods cannot be extended to the cases when the coefficients of the hyperbolic operator \eqref{eq:hyperbolic_operator} are time-dependent. The closest adaptations of the BC-method to studies in this direction were obtained recently in \cite{Alexakis_Feiz_Oksanen_22,Alexakis_Feiz_Oksanen_23}, where the authors showed that time-dependent zeroth order coefficients can be recovered from the DN map if the known globally hyperbolic Lorentzian manifold satisfies certain curvature bounds. In the ultra-static case as in the current paper, many works such as
\cite{bellassoued2011stability,feizmohammadi2021recovery,Kian_Oksanen, liu2023partial,  liu2025recovery} make use of geometric optics solutions to reduce the unique recovery of  time-dependent or the stable recovery of  time-independent lower order terms to an integral geometry problem.

Finally, let us mention that we have restricted our discussion about the BC-method only to the unique recovery of  coefficients and geometry. Several variants of the BC-method have been studied computationally \cite{belishev1999dynamical,Hoop2016,kabanikhin2004direct, oksanen2024linearized,pestov2010numerical}, and related stability questions have been investigated in \cite{anderson2004boundary, Bosi2022Reconstruction,burago2020quantitative, feizmohammadi2021recovery, korpela2016regularization, liu2025h, Liu2012}, among others. 

\subsection{Outline of the paper}

This paper is organized as follows. 
In Section \ref{sec:from_CGG_to_Exact_controllabity} we explain why the Cauchy problem \eqref{eq:Cauchy_Problem} is exactly controlled from the set $\cR$ at time $T>\diam(N,g)$ whenever $\cR$ satisfies the geometric control condition \eqref{eq:GCC}.

The purpose of Section \ref{sec:prelim} is to 
collect several key tools needed to prove Theorem \ref{thm:basic_uniqueness_thm}, which have been established in earlier works. These include the higher order approximate controllability (Proposition \ref{thm:approximate_controllability}) and Blagoveščenskii's identity (Lemma \ref{lem:Blago_identity}), as well as their immediate implications. 

Section \ref{sec:proof} is devoted to the proof of Theorem \ref{thm:basic_uniqueness_thm}, and it contains three subsections. In Subsection \ref{subsec:controllable_sets}, for each point $x_0\in N$,  we construct neighborhoods $(X_k)_{k=1}^\infty$ that collapse to $x_0$. The normalized characteristic functions $\frac{1_{X_k}}{\text{Vol}(1_{X_k})}$ are then approximated with waves $u_1^{f_{jk}}(T)$ that solve the Cauchy problem \eqref{eq:Cauchy_Problem} with $A=A_1$ and $q=q_1$, where the smooth sources $f_{jk}$ are supported in $(0,T) \times \cS$, see Propositions \ref{prop:Z-set} and \ref{prop:sets_X_k_x_0}. In this construction we utilize the local symmetry assumption \eqref{eq:adjusted_dist_function} for distances, and it will be crucial in the final parts of the proof. In Subsection \ref{subsec:correspondence} we utilize exact controllability (Proposition \ref{prop:L2_control}) to show that the waves $u^h_1$ and $u_2^h$ and their time derivatives have comparable $L^2$-norms, when the source $h$ is supported in $(0,T)\times \cR$. This is formalized in Lemma \ref{lm:exact_control}, and is the first place where we need the geometric control condition \eqref{eq:GCC}. Finally, in Subsection \ref{subsec:proof} we combine everything and prove Theorem \ref{thm:basic_uniqueness_thm}. The key argument is to show that for the sources $f_{jk}$ as above, the corresponding sequence of waves $u_2^{f_{jk}}$ converges weakly to some function $v_k$ whose norm is comparable with the volume of $X_k$. Then we apply the Lebesgue differentiation theorem together with  Blagoveščenskii's identity to first prove the local existence of a smooth non-vanishing function $\kappa$ satisfying the equation
\[
\kappa(x) u_2^\phi(T,x)=u_1^\phi(T,x).
\]
Here $x \in N$ is near $x_0$ and $\phi$ is any smooth compactly supported source function in a certain open subset of $(0,T)\times \cR$. In Lemma \ref{lem:from_equality_of_waves_to_equality_of_LOTs} we show that this equation yields the identities
\[
q_1=q_2 
\quad \text{ and } \quad 
A_1=A_2+i\kappa^{-1}d\kappa
\]
near $x_0$. 
The final piece in the proof is to show that such an equation holds globally, and that the obtained function $\kappa$ is unitary and equals to one in $\cS \cup \cR$. The verification of the former claim is part of the main proof, and the latter is justified in Lemma \ref{lm:kappa_trans_solution}.

\subsection*{Acknowledgments}
M.L. was partially supported by the Advanced Grant project
101097198 of the European Research Council, Centre of Excellence of Research
Council of Finland (grant 336786) and the FAME flagship of the Research Council
of Finland (grant 359186). B.L. was partially supported by the  Simons Foundation Travel Support for Mathematicians (MPS-TSM-00013766). 
T.S. was partially supported by the National Science Foundation (DMS-2204997, DMS-2510272) and the  Simons Foundation Travel Support for Mathematicians (MPS-TSM-00013291). A.S. was partially supported by the National Science Foundation (DMS-2204997). Z.Z. was supported by the Finnish Ministry of Education and Culture’s Pilot for Doctoral Programmes (Pilot project Mathematics of Sensing, Imaging and Modelling).

Views and opinions expressed are those of the authors only and do not necessarily reflect those of the European Union or the other funding organizations. Neither the European Union nor the other funding organizations can be held responsible for them. 

\section{From the Geometric Control Condition \eqref{eq:GCC} to the Exact Controllability of the Cauchy Problem \eqref{eq:Cauchy_Problem}}
\label{sec:from_CGG_to_Exact_controllabity}

We recall that due to the wellposedness of the Cauchy problem \eqref{eq:Cauchy_Problem}, for each $T>0$, the control operator $C$, as defined in \eqref{eq:control_op}, is a bounded linear operator.
We now introduce the following classical definition:
\begin{defn}
\label{def:exact_control}
The Cauchy problem \eqref{eq:Cauchy_Problem} is said to be exactly controllable from the set $\cR$ at time $T$ if the control operator $C$ is onto.
\end{defn}
Our aim in this section is to show that the assumption \eqref{eq:GCC} implies the exact controllability of the problem \eqref{eq:Cauchy_Problem} from $\cR$ at time $T$. To achieve this, consider the observability operator
\[
\begin{aligned}
C^* : H^{-1}(N)&\times L^2(N)\to L^2((0,T)\times \cR), 
\quad 
C^\ast(v_T,u_t):=w|_{(0,T)\times \cR},
\end{aligned}
\]
where the function $w$ solves the   terminal value problem
\begin{align}
\label{eq:dual_cauchy_Problem}
\begin{cases}
\Box w(t,x)=0 & \text{ in } (0,T)\times N,
\\
w(T,\cdot)=u_1 & \text{ in } N,
\\ 
\p_t w(T,\cdot)=v_1 & \text{ in } N.
\end{cases}
\end{align}
By the wellposedness of the hyperbolic equation, $C^*$ is a continuous operator.
Let us now the following definition of observability:
\begin{defn}
We say that the system \eqref{eq:dual_cauchy_Problem} is observable from the set $\cR$ at time $T$ if there exists a constant $C>0$ such that
\[
\norm{v_1}_{H^{-1}(N)}+\norm{u_1}_{L^2(N)}\leq C\norm{C^\ast(v_T,u_t)}_{L^2((0,T)\times \cR)}.
\]
\end{defn}
\begin{prop}
If $(N,g)$ is a closed Riemannian manifold, and an open set $\cR\subset N$ satisfies the assumption \eqref{eq:GCC},  the system \eqref{eq:dual_cauchy_Problem} is observable.
\end{prop}
\begin{proof}
See for instance \cite{Bardos92}, \cite[Remark 1.1 and Theorem 1.2] {LLTT17}, or \cite[Section 3]{Rauch74}.
\end{proof}

We claim that the observability operator $C^*$ is the adjoint of the control operator $C$. Indeed, for any function $f\in L^2((0,T)\times \cR)$, we denote by $\tilde f$  the zero extension of $f$ to $(0,T)\times N$. Let  $(v_1,u_1)\in C^\infty(N)\times C^\infty(N)$,
then due to the initial conditions in the Cauchy problem \eqref{eq:Cauchy_Problem},  after integrating by parts, we have
\[
w(T,\cdot)\p_t u^{\widetilde{f}}(T,\cdot) =  \int_0^T   \p_t^2 w(t,\cdot)  u^{\widetilde{f}}(t,\cdot) dt+ \int_0^T \p_t w(t,\cdot) \p_t u^{\widetilde{f}}(t,\cdot) dt,
\]
and
\[
\p_t w(T,\cdot) u^{\widetilde{f}}(T,\cdot) = \int_0^T   w(t,\cdot)\p_t^2 u^{\widetilde{f}}(t,\cdot) dt+ \int_0^T \p_t w(t,\cdot) \p_t u^{\widetilde{f}}(t,\cdot) dt.
\]
Hence, it follows that
\[
\begin{aligned}
\inner{C(f),(v_1,u_1)}
&:=
\inner{-u^{\widetilde{f}}(T,\cdot),v_1}+\inner{\p_t u^{\widetilde{f}}(T,\cdot),u_1}
\\
&=
\inner{-u^{\widetilde{f}}(T,\cdot),\p_t w(T,\cdot)}+\inner{\p_t u^{\widetilde{f}}(T,\cdot),w(T,\cdot)}
\\
&=
\int_0^T \inner{w(t,\cdot),\p_t^2 u^{\widetilde{f}}(t,\cdot)}-\inner{\p_t^2 w(t,\cdot),u^{\widetilde{f}}(t,\cdot)} dt
\\
&=
\int_0^T \inner{w(t,\cdot),\widetilde{f}-\L u^{\widetilde{f}}(t,\cdot)}+\inner{\L w(t,\cdot),u^{\widetilde{f}}(t,\cdot)} dt
\\
&=
\int_0^T \inner{w(t,\cdot),\L u^{\widetilde{f}}(t,\cdot)+\widetilde{f}}
-
\inner{w(t,\cdot),\L u^{\widetilde{f}}(t,\cdot)} dt
\\
&=
\int_0^T \inner{w(t,\cdot),f}\, dt
\\
&=\inner{f,C^*(v_1,u_1)}_{L^2((0,T)\times\cR)}.
\end{aligned}
\]
For any $(v_T,u_T)\in H^{-1}(N)\times L^2(N)$, we can find a sequence $\{(v_i,u_i)\}_{i=1}^\infty$ in $C^\infty(N)\times C^\infty(N)$ such that $(v_i,u_i)\to (v_T,u_T)$ in $H^{-1}(N)\times L^2(N)$ as $i\to\infty$. Hence, by the observability of the system \eqref{eq:dual_cauchy_Problem}, we have
\begin{align*}
\inner{C(f),(v_T,u_T)}&=\lim_{i\to\infty}\inner{C(f),(v_i,u_i)}=\lim_{i\to\infty}\inner{f,C^*(v_i,u_i)}_{L^2((0,T)\times\cR)}\\
&=\inner{f,C^*(v_T,u_T)}_{L^2((0,T)\times\cR)}.
\end{align*}
This proves our claim.

Since the operators $C$ and $C^\ast$ are adjoint to each other, according to \cite[Theorem 2.1]{Dolecki77},  the system \eqref{eq:Cauchy_Problem} is exactly controllable from the set $\cR$ at time $T$ if and only if the system \eqref{eq:dual_cauchy_Problem} is observable from the set $\cR$ at time $T$. To summarize, we have the following result.

\begin{prop}
\label{prop:L2_control}
Let $(N,g)$ be a smooth closed Riemannian manifold and $T>\diam(N,g)$. Let $\cR\subset N$ be an open set that satisfies the property \eqref{eq:GCC}.
If $A$ is a real-valued smooth co-vector field and $q$ is a smooth real-valued function in $N$,  the control operator $C$, as defined in \eqref{eq:control_op}, is onto.
\end{prop}

In the absence of the lower order terms $A$ and $q$, the magnetic Shr\"odinger operator $\L$ is indeed the usual Laplace-Beltrami operator $-\Delta_g$ of $(N,g)$. In this case, it is established in \cite[Theorem 1]{HPT19} that the exact controllability of the wave equation from the set $\cR$ at time $T$ implies the existence of a constant $C>0$ such that
\begin{equation}
\label{eq:quantum_chaos}
C<\inf_{\phi\in \mathcal{E}}\frac{\norm{\phi}_{L^2(\cR)}}{\norm{\phi}_{L^2(N)}},
\end{equation}
where $\mathcal{E}$ is the set of all non-constant eigenfunctions   of $-\Delta_g$. We note that, in light of the example provided in \cite[p.753]{HPT19}, the condition \eqref{eq:quantum_chaos} is strictly weaker than exact controllability. Moreover, the authors of \cite{Lassas_Nul_Oksanen_Ylinen} showed that when the condition \eqref{eq:quantum_chaos} is satisfied, one can determine the isometry type of a closed Riemannian manifold from the local source-to-solution map associated with the problem \eqref{eq:Cauchy_Problem} with disjoint source and receiver sets. It is also established in \cite[Lemma 5.1]{feizmohammadi2024calderonproblemfractionalschrodinger} that \eqref{eq:quantum_chaos} is satisfied when $N\setminus\cR$ is nontrapping, i.e. all geodesics that start in $N\setminus\cR$ exit in finite time.

\section{Preliminaries}
\label{sec:prelim}
In this short section we introduce several necessary  key tools that have been established in some earlier works  to prove the main theorem. 
For each open set $\mathcal{V} \subset N$ and $T>0$, we define the respective domain of influence as
\[
M(\mathcal{V},T):=\{x\in  N\mid \inf_{y\in \V}d(x,y)\leq T\}.
\]

By $H_0^k(M(\mathcal{V},T))$ and $k \in \{0,1,\ldots\}$, we mean the completion of $C_0^\infty(M(\mathcal{V},T)^\circ$ with respect to $\norm{\cdot}_{H^k(N)}$-norm. Here the notation $A^\circ$ stands for the interior of a subset $A$ of $N$.
We first introduce an approximate controllability result for higher order Sobolev spaces. For $k$ large enough, this yields the existence of waves that do not vanish at a given point in space.

\begin{prop}[Higher order approximate controllability]
\label{thm:approximate_controllability}

For any $T > 0$ and $k \in \{0,1,\ldots\}$, the set
\begin{align*}
\mathcal{W}_T := 
\{w^f(T,\cdot) \in C^\infty(N):  
w^f \text{ solves the problem \eqref{eq:Cauchy_Problem} with source }
f\in C_0^\infty((0,T)\times\mathcal{S})
\}
\end{align*}
is dense in $H_0^k(M(\mathcal{S},T))$ with respect to the $H^k(N)$-topology.

In particular, for each $f \in C_0^\infty((0,T)\times\mathcal{S})$ we have 
\[
\text{supp} \ (u^f(T,\cdot)) \subset M(\mathcal{S},T)^\circ.
\]
\end{prop}

\begin{proof}
This result follows from the finite speed of wave propagation and the unique continuation principle from\cite{tataru1995unique}. We refer readers to  \cite[Corollary 2.15 and Theorem 2.26]{saksala2025inverse} for the full proof.
\end{proof}

\begin{rem}
The $L^2$-version of the approximate controllability was established in \cite{source-to-solution} for the solutions of the standard wave equation. 
\end{rem}

\begin{cor}[Nonvanishing wave]
\label{cor:nonvanishing_condition}
Let $T>0$, $\cS \subset N$ be an open set, and let $x_0 \in N$ be such that $d_g(x_0,\cS) < T$. Then there exists a function $f\in C_0^\infty((0,T)\times \cS)$ such that the corresponding smooth solution $u^f$ of the Cauchy problem \eqref{eq:Cauchy_Problem} does not vanish at $(T,x_0)$.
\end{cor}

\begin{proof}
See \cite[Corollary 2.28]{saksala2025inverse} for the proof.
\end{proof}

We also need the following important identity that links the solutions of   two different hyperbolic PDEs together in a certain global sense if Hypothesis \eqref{eq:hypothesis} is valid. This identity, originating from \cite{blagoveshchenskii1971inverse, blagovestchenskii1969one}, is crucial to establish  Theorem \ref{thm:basic_uniqueness_thm}.

\begin{lem}[Blagoveščenskii's identity]
\label{lem:Blago_identity}
Let $(N,g)$ be a smooth closed Riemannian manifold. Let $\cS,\cR \subset N$ be open sets. For $i\in \{1,2\}$, let $A_i$ be a smooth real-valued co-vector field, and let $q_i$ be a smooth real-valued function. If   
$\Lambda_{\cS,\cR,2T}^{A_1,q_1}=\Lambda_{\cS,\cR,2T}^{A_2,q_2}$,  the equation
\begin{equation}
\label{eq:Blago}
(u_1^f(T,\cdot),u_1^h(T,\cdot))_{L^2(N)}
=
(u_2^f(T,\cdot),u_2^h(T,\cdot))_{L^2(N)}
\end{equation}
holds for all functions $f\in C^\infty_0((0,2T)\times \cS)$ and $h\in C^\infty_0((0,2T)\times \cR)$, where $u_i^f$ and $u^h_i$ are the solutions to the Cauchy problem \eqref{eq:Cauchy_Problem} with $A=A_i$ and $q=q_i$ and sources $f$ and $h$, respectively. 
\end{lem}

\begin{proof}
Since the magnetic Schr\"odinger operators $\L$ for $A=A_i$ and $q=q_i$ defined in \eqref{eq:mag_schro_op} are self-adjoint, the equation \eqref{eq:Blago} can be derived by following the same arguments as in \cite[Lemmas 3.1 and 3.2]{Lassas_Nul_Oksanen_Ylinen}. Therefore, we shall omit the details.
\end{proof}

Moreover, we also have the Blagoveščenskii's identities for $L^2$-regular sources as follows.

\begin{lem}
\label{lem:new_blago}
Let $f\in L^2((0,2T)\times \cS)$ and $h\in L^2((0,2T)\times \cR)$. Under the same hypotheses as in Lemma \ref{lem:Blago_identity}, we have  
\begin{equation}
\label{eq:new_blago_id}
\begin{aligned}
\inner{u_1^f(T,\cdot), u_1^h(T,\cdot)}=\inner{u_2^f(T,\cdot), u_2^h(T,\cdot)} 
\quad \text{ and } \quad 
\inner{u_1^f(T,\cdot),\p_t u_1^h(T,\cdot)}=\inner{u_2^f(T,\cdot),\p_t u_2^h(T,\cdot)}.
\end{aligned}
\end{equation}
\end{lem}
\begin{proof}
We only present a proof for the latter equation, as the former follows from a simpler version of the same argument.

Let us first notice that $C_0^\infty((0,2T)\times \cS)$ and $C_0^\infty((0,2T)\times \cR)$ are dense in $L^2((0,2T)\times\cS)$ and $L^2((0,2T)\times\cR)$, respectively. Then we can find sequences $\{\phi_k\}_{k=1}^\infty\subset C_0^\infty((0,2T)\times \cS)$ and $\{\psi_j\}_{j=1}^\infty\subset C_0^\infty((0,2T)\times \cR)$ such that $\phi_k\to f$ in $L^2((0,2T)\times\cS)$ and $\psi_j\to h$ in $L^2((0,2T)\times\cR)$.
Due to the wellpossedness and the linearity of the Cauchy problem \eqref{eq:Cauchy_Problem}, we obtain for both $i \in \{1,2\}$ that
\[
\norm{u_i^f(T,\cdot)-u_i^{\phi_k}(T,\cdot)}_{L^2(N)}
\leq C_1 \norm{f-\phi_k}_{L^2((0,2T)\times\cS)}
\]
and 
\[
\norm{\p_t u_i^h(T,\cdot)-\p_t u_i^{\psi_j}(T,\cdot)}_{L^2(N)}\leq C_2 \norm{h-\psi_j}_{L^2((0,2T)\times\cR)},
\]
for some uniform constants $C_1,C_2>0$. Therefore, as $k \to \infty$ and $j \to \infty$, it follows that $u_i^{\phi_k}(T,\cdot)\to u_i^{f}(T,\cdot)$ and $\p_t u_i^{\psi_j}(T,\cdot)\to \p_t u_i^h(T,\cdot)$ in $L^2(N)$ for both $i \in \{1,2\}$.

We also observe that if $\psi_j \in C_0^\infty((0,2T)\times \cR)$, then the smooth function $\p_t u_i^{\psi_j}$ solves the Cauchy problem \eqref{eq:Cauchy_Problem} with the interior source $\p_t \psi_j \in C_0^\infty((0,2T)\times \cR)$. Indeed, since the coefficients $A_i$ and $q_i$ are time-independent, we differentiate the equation 
\begin{equation}
\label{eq:wave_eq_for_u_psi}
\square_{A_i,q_i}  u_i^{\psi_j}= \psi_j \quad \text{ in } (0,2T)\times N
\end{equation}
with respect to $t$ to get $
\square_{A_i,q_i} (\p_t u_i^{\psi_j})=\p_t \psi_j
$. Moreover, due to the initial conditions posed in \eqref{eq:Cauchy_Problem}, we have $u_i^{\psi_j}(0,\cdot)=\p_t u_i^{\psi_j}(0,\cdot)=0$. Hence, we get from \eqref{eq:wave_eq_for_u_psi} and the assumption $\psi_j \in C_0^\infty((0,2T)\times \cR)$ that  $\p_t^2 u_i^{\psi_j}(0,\cdot)=0$. 

Therefore, Lemma \ref{lem:Blago_identity} yields that
\[
\inner{u_1^{\phi_k}(T,\cdot),\p_t u_1^{\psi_j}(T,\cdot)}_{L^2(N)}
=
\inner{ u_2^{\phi_k}(T,\cdot),\p_t u_2^{\psi_j}(T,\cdot)}_{L^2(N)} \quad \text{ for all } k,j \in \N.
\]
From here, we obtain the identity \eqref{eq:new_blago_id}  by taking $k,j \to \infty$ in the equation above. This completes the proof of Lemma \ref{lem:new_blago}.
\end{proof}

\section{Proof of Theorem \ref{thm:basic_uniqueness_thm}}
\label{sec:proof}

In this section we prove Theorem \ref{thm:basic_uniqueness_thm} via utilizing the tools introduced in the previous sections. This section contains three subsections. In the first subsection we show that each point of the manifold has a sequence of neighborhoods that collapse to this point. Due to the local symmetry assumption for the distances \eqref{eq:assumption_rev}, we can carry over the construction of these sets so that with the aid of exact controllability, which follows from the assumption \eqref{eq:GCC}, we can construct a sequence of waves that concentrate to a delta source centered at the chosen point in the third subsection. This allows us to construct the gauge function $\kappa$ promised in Theorem \ref{thm:basic_uniqueness_thm}. To achieve this, we also need to show that the $L^2$-norms of the waves $u_1^h(T,\cdot)$ and $u_2^h(T,\cdot)$ with a source $h \in L^2((0,T)\times \cR)$ are comparable whenever $T>\diam (N,g)$. This is accomplished in the second subsection.

\subsection{Definition and properties of controllable sets}
\label{subsec:controllable_sets}

Throughout this subsection we shall fix an arbitrary point $x_0 \in N.$
Let us recall that the cut locus $Cut(x_0) \subset N$ of  $x_0$ is the closure of all the points of $N$ that can be connected to $x_0$ via more than one distance minimizing geodesic.  
Since $Cut(x_0)$ of   $x_0$ is a closed set of measure zero \cite[Theorem 10.34]{lee_riemannian} and the set  $\cR$ is open, we can choose $y\in \cR\setminus Cut(x_0)$ and write $s=d(x_0,y)$. Then there exists $\varepsilon>0$ and a unique distance minimizing geodesic $\gamma: [-2\varepsilon,s+\varepsilon]\to N$ such that $\gamma(0)=y$, $\gamma(s)=x_0$. Moreover, for all $\delta\in (0,\varepsilon]$ we have
\begin{enumerate}
\item $\gamma(-2\delta)\not\in Cut(x_0)$,
\item $\gamma(s+\delta)\not\in Cut(y)$,
\item $\overline{B(\gamma(-\delta),\delta)}\subset \cR$.
\end{enumerate}

For each $\delta\in (0,\varepsilon]$, we define the set
\begin{equation}
\label{eq:def_Z_delta}
Z_\delta:=\overline{B(\gamma(\varepsilon-\delta),s-\varepsilon+2\delta)\setminus B(\gamma(-\varepsilon-\delta),s+\varepsilon-\delta)}.
\end{equation}
Since $x_0=\gamma(s)$ and $y=\gamma(0)$, it holds that $x_0$ is contained in 
$
B(\gamma(\varepsilon-\delta),s-\varepsilon+2\delta),
$
but not in  
$
B(\gamma(-\varepsilon-\delta),s+\varepsilon-\delta).
$
Hence, $x_0$ is contained in the interior $Z^\circ_\delta$ of the set $Z_\delta$.

The following proposition summarizes some further properties of the sets $Z_\delta$. 
\begin{prop}
\label{prop:Z-set}
Let $x_0,y,s$, and $\varepsilon$ be as above. Then for all $\delta\in (0,\varepsilon]$, the sets $Z_\delta$ satisfy the following properties:
\begin{enumerate}
\item $B(x_0,\delta)\subset Z_\delta$ for all $\delta\in (0,\varepsilon]$,
\item $\diam(Z_\delta)\to 0$ as $\delta\to 0$.
\end{enumerate}
Here $\diam(A)$ stands for the diameter of a set $A \subset N$.
\end{prop}

\begin{proof}
We shall begin by verifying the property (1). Let $\delta \in (0,\varepsilon]$ and $p\in B(x_0,\delta)$. Then we have  
\[
d(p,\gamma(\varepsilon-\delta))\leq d(\gamma(\varepsilon-\delta),x_0)+d(p,x_0)< s-\varepsilon+2\delta.
\]
Thus, it follows that $p \in B(\gamma(\varepsilon-\delta),s-\varepsilon+2\delta)$. On the other hand, we also deduce that
\[
d(p,\gamma(-\varepsilon-\delta))\geq d(\gamma(-\varepsilon-\delta),x_0)-d(x_0,p)> s+\varepsilon.
\]
Hence, $p \notin B(\gamma(-\varepsilon-\delta),s+\varepsilon-\delta)$. Since $p\in B(x_0,\delta)$ was arbitrarily chosen, the property (1) holds. 

We next show that 
\begin{equation}
\label{eq:Z_increase}
Z_\delta\subset Z_{\delta^\prime} \text{ if } \delta\leq \delta^\prime \leq \varepsilon.
\end{equation}
In order to prove \eqref{eq:Z_increase}, it suffices to show that 
\begin{equation}
\label{eq:ball_containment}
B(\gamma(-\varepsilon-\delta^\prime),s+\varepsilon-\delta^\prime)\subset B(\gamma(-\varepsilon-\delta),s+\varepsilon-\delta),
\end{equation}
and 
\begin{equation}
\label{eq:ball_containment_2}
B(\gamma(\varepsilon-\delta),s-\varepsilon+2\delta)\subset B(\gamma(\varepsilon-\delta^\prime),s-\varepsilon+2\delta^\prime).
\end{equation}
Indeed, let $p\in B(\gamma(-\varepsilon-\delta^\prime),s+\varepsilon-\delta^\prime)$. Then we have
\begin{align*}
d(p,\gamma(-\varepsilon-\delta))
&\leq
d(p,\gamma(-\varepsilon-\delta'))+d(\gamma(-\varepsilon-\delta'),\gamma(-\varepsilon-\delta))
\\
&\leq 
d(p,\gamma(-\varepsilon-\delta^\prime))+(\delta^\prime-\delta)
<s+\varepsilon-\delta^\prime+\delta^\prime-\delta=s+\varepsilon-\delta.
\end{align*}
Similarly, for any $p\in  B(\gamma(\varepsilon-\delta),s-\varepsilon+2\delta)$, we compute directly that
\begin{align*}
d(p,\gamma(\varepsilon-\delta^\prime))
&\leq d(p,\gamma(\varepsilon-\delta))+d(\gamma(\varepsilon-\delta),\gamma(\varepsilon-\delta^\prime))
\\
&\leq 
d(p,\gamma(\varepsilon-\delta))+(\delta^\prime-\delta)<s-\varepsilon+\delta+\delta^\prime<s-\varepsilon+2\delta^\prime.
\end{align*}
Therefore,  \eqref{eq:ball_containment} and \eqref{eq:ball_containment_2} follow immediately from the previous two inequalities. Furthermore, the condition \eqref{eq:Z_increase} is also valid.

We now define two auxiliary sets 
\[
Z_\delta^+:= \overline{B(\gamma(\varepsilon-\delta),s-\varepsilon+2\delta)\setminus B(y,s)},
\]
and 
\begin{equation}
\label{eq:Z_delta}
Z_\delta^-:=\overline{B(y,s)\setminus B(\gamma(-\varepsilon-\delta),s+\varepsilon-\delta)}.
\end{equation}
Due to the properties of closure, we have  
\begin{equation}
\label{eq:Z_delta_as_union}
Z_\delta\subset Z_\delta^+\cup Z_\delta^-.
\end{equation}

To complete the proof, we assume for a moment that
\begin{enumerate}
\item[(i)] $Z_\delta^-\subset Z_{\delta^\prime}^-$ if $\delta\leq \delta^\prime\leq \varepsilon$.
\item[(ii)] $\bigcap_{0<\delta\leq \varepsilon} Z_\delta^-=\{x_0\}$.
\end{enumerate}
Thus, by combining (i) and (ii) with the analogous properties for the sets $Z_\delta^+$, as proven in \cite[Proposition 4.7 and Lemma 4.8]{saksala2025inverse}, we obtain from \eqref{eq:Z_delta_as_union}  that
\[
\bigcap_{0<\delta\leq \varepsilon} Z_\delta\subset \bigcap_{0<\delta\leq \varepsilon}(Z_\delta^+\cup Z_\delta^-)=\{x_0\}.
\]
Notice that Property (1)   implies that $x_0\in Z_\delta$ for all $0<\delta\leq\varepsilon$, hence we get  
\begin{equation}
\label{eq:Z_limit}
\bigcap_{0<\delta\leq \varepsilon} Z_\delta=\{x_0\}.
\end{equation}

Armed with \eqref{eq:Z_limit}, we proceed to prove Property (2).
Clearly, it follows from \eqref{eq:Z_increase} that the function $\delta \mapsto \diam(Z_\delta)$ is increasing. Thus, there exists $\eta\geq 0$ such that $\lim_{\delta\to 0}\diam(Z_\delta)=\eta$. For the sake of contradiction, we assume that $\eta>0$ and write $\widetilde{Z_\delta}=Z_\delta\setminus B(x_0,\frac{\eta}{3})$. 
Since $\diam B(x_0,\frac{\eta}{3})\leq \frac{2\eta}{3}<\eta$, we see that $\widetilde{Z_\delta}\neq \emptyset$ for all $0<\delta\leq \varepsilon$. However, there exists $K \in \N$ such that $\frac{1}{K}<\varepsilon$, and according to the nested compact set property, see for instance \cite[Theorem 26.9]{Munkres}, we get $\bigcap_{k\geq K} \widetilde{Z_{1/k}}\neq \emptyset$. However, as we also have $\widetilde{Z_\delta} \subset Z_\delta$,   the property \eqref{eq:Z_limit} is violated. 
Therefore, we must have  $\eta=0$. This proves Property (2).

In order to finish the proof of the proposition, we next verify Properties (i) and (ii) above.
Since the ball $B(y,s)$ is fixed, Property (i) follows immediately from \eqref{eq:ball_containment}.

To prove (ii), we first show that 
\begin{equation}
\label{eq:1}
B(y,s)\subset \bigcup_{0<\delta\leq \varepsilon} B(\gamma(-\varepsilon-\delta),s+\varepsilon-\delta).
\end{equation}
If  $p \in B(y,s)$, we can find $\xi>0$ such that $d(p,y)<s-\xi$, and we get from $y=\gamma(0)$ that
\[
d(p,\gamma(-\varepsilon-\delta))\leq d(p,y)+\varepsilon+\delta<s-\xi+\varepsilon+\delta.
\]
Whenever $\delta<\frac{\xi}{2}$, the previous estimate yields that $d(p,\gamma(-\varepsilon-\delta))<s+\varepsilon-\delta$. Therefore, we conclude that  $p\in B(\gamma(-\varepsilon-\delta),s+\varepsilon-\delta)$. Since $p$ is arbitrary, we obtain \eqref{eq:1}.

On the other hand, we get from \eqref{eq:Z_delta} that 
\[
Z_\delta^-\subset \overline{B(y,s)}=B(y,s)\cup \p B(y,s), \quad \text{for every } \delta\in (0,\epsilon].
\]
Meanwhile, \eqref{eq:1} implies that
\begin{equation}
\label{eq:inclusion_of_Z_delta_intersection}
\bigcap_{0<\delta\leq \varepsilon}Z_\delta^- \subset \p B(y,s).
\end{equation}

To finish the proof of property (ii), it remains to verify that
\[
\bigcap_{0<\delta\leq \varepsilon}Z_\delta^- \cap \p B(y,s)=\{x_0\}.
\]
To this end, we recall that all small metric balls are strongly convex, see for instance \cite[Chapter 3, Proposition 4.2]{doCarmo}, in the sense that for any two points contained in the closure of the ball, there is a unique distance minimizing geodesic contained in the ball, possibly excluding the terminal points. Let $\zeta \in (0,\varepsilon]$ be sufficiently small so that $B(y,\zeta)$ is  strongly convex. Let $z\in \p B(y,s)$, $z\neq x_0$, and let us recall that on complete Riemannian manifolds we have 
\[
\p B(y,s)=\{x \in N: \: d(x,y)=s\}.
\]
By the Hopf-Rinow theorem, there exists a geodesic $\gamma_z$ such that $\gamma_z(0)=y$ and $\gamma_z(s)=z$. Due to strong convexity, the points $\gamma_z(\zeta)$ and $\gamma(\zeta)$ cannot coincide. Furthermore, we get from the strong convexity of $B(y,\zeta)$ that $d(\gamma_z(\zeta),\gamma(-\zeta))<2\zeta$. If this was not the case, we would need to have  $d(\gamma_z(\zeta),\gamma(-\zeta))=2\zeta$, and the curve $\gamma_z|_{[0,\theta]} \circ \gamma|_{[-\zeta,0]}$ would be distance minimizing since it has length $2\zeta$. However, this is not possible as $\gamma_z(\zeta)\neq \gamma(\zeta)$ implies that the curve $\gamma_z|_{[0,\theta]} \circ \gamma|_{[-\zeta,0]}$ is not smooth at $0$. Hence, this curve cannot minimize distances. 

For small enough $\sigma>0$, we can write $d(\gamma_z(\zeta),\gamma(-\zeta))<2\zeta-\sigma$. Hence, we have
\begin{align*}
d(z,\gamma(-\varepsilon-\delta))&\leq d(z,\gamma_z(\zeta))+d(\gamma_z(\zeta),\gamma(-\zeta))+d(\gamma(-\zeta),\gamma(-\varepsilon-\delta))\\
&<s-\zeta+2\zeta-\sigma+\varepsilon+\delta-\zeta=s+\varepsilon+\delta-\sigma.
\end{align*}
Thus, whenever $\delta<\sigma/2$, we conclude that $z\in B(\gamma(-\varepsilon-\delta),s+\varepsilon-\delta)$ and $z \notin Z_\delta^-$. Therefore, \eqref{eq:inclusion_of_Z_delta_intersection} implies that $\bigcap_{0<\delta\leq \varepsilon}Z_\delta^- \subset \{x_0\}$, and the equality follows from Property (i) and the compact nested set property. This completes the proof of Proposition \ref{prop:Z-set}.
\end{proof}

We say that an open set $U\subset N$ satisfies the condition (C) if the following holds:
\begin{equation}
\label{con:convex}
\tag{C}
\begin{split}
&\text{For all distinct }y,z\in U, \text{ there are }p,q\in \cR, \\
&\text{such that }d(p,z)<d(p,y)\text{ and }d(q,y)<d(q,z).
\end{split}
\end{equation}
By the assumption \eqref{eq:assumption_rev}, every point of $N$ has a neighborhood satisfying the condition \eqref{con:convex}.

Let $\mathcal{V} \subset N$ be open. For a continuous function $h: \overline{\mathcal{V}}\to \R$, we define a modified domain of influence as
\begin{equation}
\label{eq:adjusted_domain_of_influence}
M(\mathcal{V},h):=\{x\in  N\mid \inf_{y\in \V}(d(x,y)-h(y))\leq 0\}.
\end{equation}
In what follows we shall always choose the function $h$ to be one of the following modified distance functions: 
\begin{equation}
\label{eq:adjusted_dist_function}
h_{k,x_0}(z)=d(x_0,z)-\frac{1}{k},\ z\in \mathcal{V}, \: k \in \N.
\end{equation}

\begin{prop}
\label{prop:sets_X_k_x_0}
Let $x_0\in N$, and let $U$ be an open neighborhood of $x_0$ satisfying the condition \eqref{con:convex}. For each $k \in \N$ and $h_{k,x_0}$ defined in \eqref{eq:adjusted_dist_function} with domain $\mathcal{V}=\cR$, the set 
\begin{equation}
\label{eq:set_X_{k,x}}
X_{U,k,x_0}:=U \setminus M(\cR, h_{k,x_0})
\end{equation} 
is an open neighborhood of $x_0$ and $\diam(X_{U,k,x_0})\to 0$ as $k\to \infty$. 
\end{prop}
\begin{proof}
Since the set $U$ and point $x_0$ remain fixed throughout the proof, we will use the shorthand $X_k := X_{U,k,x_0}$ to simplify the notation.
Since $U$ is an open neighborhood of $x_0$, there exists an open ball $B(x_0,\delta)\subset U$ whenever $\delta>0$ is small enough. Let $k \in \N$, and let us write  $\eta:=\min\{\delta,1/k\}$. Due to the triangle inequality, for any $p\in B(x_0,\eta)$ and $z\in \cR$, we have  
\[
d(p,z)\geq d(x_0,z)-d(x_0,p)>d(x_0,z)-\frac{1}{k}.
\]
Hence, we obtain $p\not\in  M(\cR, h_{k,x_0})$ and $B(x_0,\eta)\subset X_k$. This proves the first claim in the proposition.

Turning our attention to the second claim, it follows from \eqref{eq:adjusted_domain_of_influence} and $\eqref{eq:adjusted_dist_function}$ that $x_0 \notin M(\cR,h_{k,x_0})$ for any $k \in \N$. Since $U$ is a neighborhood of $x_0$, due to \eqref{eq:set_X_{k,x}}, we see that  $x_0 \in \overline{X_k}$ for all $k \in \N$. 

We now aim to show that $\bigcap_{k=1}^\infty \overline{X_k}= \{x_0\}$. If $y \in U \setminus \{x_0\}$, thanks to the condition \eqref{con:convex}, we can choose $z_y \in \cR$ such that $d(y,z_y)<d(x_0,z_y)$. Thus, we can also choose
$0<\rho< \frac{1}{2}(d(x_0,z_y)-d(y,z_y))$ and take any $p\in B(y,\rho)$. Then it follows that
\[
d(p,z_y) < d(y,z_y)+\rho<\frac{1}{2}(d(x_0,z_y)+d(y,z_y)).
\]
If we choose $k \in \N$ such that $\frac{1}{k}<\rho$, we get from the estimates above that $d(p,z_y)<d(x_0,z_y)-\frac{1}{k}$. As $p \in B(y,\rho)$ was arbitrarily chosen, we see that $B(y,\rho)\subset M(\cR,h_{k,x_0})$, which implies that $y \notin \overline{X_k}$ for sufficiently large $k$. Hence, we have $\bigcap_{k=1}^\infty \overline{X_k}= \{x_0\}$.

Since  $\bigcap_{k=1}^\infty \overline{X_k}= \{x_0\}$ and the sets  $\overline{X_k}$ are closed, we get from  \cite[Theorem 26.9]{Munkres} that $\diam(\overline{X_k})\to 0$ as $k\to \infty$. Furthermore, as $\diam(X_k)\leq \diam(\overline{X_k})$, we conclude that $\diam(X_k)\to 0$ as $k\to \infty$ as well. This completes the proof of Proposition \ref{prop:sets_X_k_x_0}.
\end{proof}

\subsection{Correspondence between the solutions of two systems}
\label{subsec:correspondence}

In this subsection we provide some key lemmas about the correspondence between the solutions of the $\Box u^f=f$, when $A=A_i$ and $q=q_i$ for $i\in \{1,2\}$ and $f \in C^\infty_0((0,T)\times \cR)$ is a common interior source. 

\begin{lem}
\label{lm:exact_control} 
Let $T>\diam (N,g)$. If the assumption \eqref{eq:GCC} and the hypothesis \eqref{eq:hypothesis} are valid,  there exists a uniform constant $C>0$ such that
\begin{equation}
\label{eq:lm_exact}
\norm{\p_t u_2^f(T,\cdot)}_{L^2(N)}
\leq
C\norm{\p_t u_1^f(T,\cdot)}_{L^2(N)}
\end{equation}
for all $f \in L^2((0,T)\times\cR)$. Moreover, the similar estimate holds when the roles of $u_1^f$ and $u_2^f$ are interchanged.
\end{lem}
\begin{proof}
For $k \in \{1,2\}$ we define the sets
\[
F_k:= \{f\in L^2((0,\infty)\times\cR)\mid \p_t u_k^f(T,\cdot)=0\},
\]
and aim to show that $F_1=F_2$. To this end. suppose that $f\in F_1$. By  Lemma \ref{lem:new_blago}, we have  
\[
\inner{u_2^h(T,\cdot),\p_t u_2^f(T,\cdot)}_{L^2(N)}=\inner{u_1^h(T,\cdot),\p_t u_1^f(T,\cdot)}_{L^2(N)}=0,\: \text{ for any } h\in C^\infty_0((0,\infty)\times\cS).
\]
Since $T> \diam(N)$, it follows that the domain of influence $M(\mathcal{S},T)$ equals to $N$. Hence, according to Proposition \ref{thm:approximate_controllability},  we  conclude that $\p_t u_2^f(T,\cdot)=0$. Consequently, we get $f\in F_2$, which yields $F_1\subset F_2$. By an analogous argument, we also have $F_2 \subset F_1$. Hence, $F_1=F_2$, and in what follows we shall simplify the notations by denoting $F=F_1$. 
Furthermore, we get from the wellposedness of   hyperbolic equations, see for instance \cite[Theorem 2.30]{KKL}, that
\begin{equation}
\label{eq:regularity}
\norm{\p_t u_1^f(T,\cdot)}_{L^2(N)} \leq C \norm{f}_{L^2((0,T)\times \cR)}.
\end{equation}
Thus, the set $F$ is $L^2$-closed as it is the kernel of a continuous operator $L^2((0,T)\times \cR) \ni f \mapsto \p_t u^f_1(T,\cdot) \in L^2(N)$. 

We now denote the quotient space $L^2((0,\infty)\times \cR)\setminus F$ by $\mathbb{F}$, and equip it with the norm induced by the quotient map
\[
\norm{f}_{\mathbb{F}}:=\norm{f-P_F(f)}_{L^2((0,\infty)\times \cR)},
\]
where $P_{F}$ is the projection operator onto $F$. According to \cite[Proposition 11.8]{Brezis}, $\mathbb{F}$ equipped with the quotient norm $\norm{\cdot}_{\mathbb{F}}$ is a Banach space. 

After these preparations, we proceed to define an injective operator
\[
L: \mathbb{F} \to L^2(N),\quad L(f)=\p_t u_1^f(T,\cdot),
\]
which is continuous due to the inequality \eqref{eq:regularity}. Furthermore, Proposition \ref{prop:L2_control}  implies that the map $L$ is also surjective. 
Therefore, it follows from the bounded inverse theorem that the inverse operator $L^{-1}\colon L^2(N) \to \mathbb{F}$ is also continuous. 
Consequently, for each $f\in L^2((0,\infty)\times\cR)$, we have
\begin{equation}
\label{eq:exact_eq_1}
\norm{f-P_{F}(f)}_{L^2((0,\infty)\times \cR)}=\norm{f}_{\mathbb{F}}
\leq
C \norm{\p_t u_1^f(T,\cdot)}_{L^2(N)}.
\end{equation}

Finally, we get from the equation $\p_t u_2^{P_{F}(f)}(T,\cdot)=0$, the wellposedness of the problem \eqref{eq:Cauchy_Problem}, and the linearity of the respective source-to-solution operator that 
\begin{equation}
\label{eq:exact_eq_2}
\norm{\p_t u_2^{f}(T,\cdot)}_{L^2(N)}=\norm{\p_t u_2^{f-P_{F}(f)}(T,\cdot)}_{L^2(N)}\leq C \norm{f-P_{F}(f)}_{L^2((0,\infty)\times \cR)}.
\end{equation}
Hence, the estimate \eqref{eq:lm_exact} follows immediately by combining the estimates \eqref{eq:exact_eq_1} and \eqref{eq:exact_eq_2}. This completes the proof of Lemma \ref{lm:exact_control}.
\end{proof}

\begin{cor}
\label{cor:smooth_control}
Let $T>\diam (N,g)$. If the assumption \eqref{eq:GCC} and the hypothesis \eqref{eq:hypothesis} are valid,  there exists a uniform constant $C>0$ such that
\begin{equation}
\label{eq:exact_no_derivative}
\norm{u_2^f(T,\cdot)}_{L^2(N)}
\leq
C\norm{u_1^f(T,\cdot)}_{L^2(N)}
\end{equation}
for all $f \in C_0^\infty((0,T)\times\cR)$. Moreover, the similar estimate holds when the roles of $u_1^f$ and $u_2^f$ are interchanged.
\end{cor}
\begin{proof}
Since $f\in C_0^\infty((0,T)\times \cR)$, we get from the fundamental theorem of calculus that there exists a function  $\widetilde{f}\in C^\infty((0,T)\times\cR) \cap L^2((0,T)\times \cR))$ such that $\p_t\widetilde{f}=f$. Moreover, as discussed in the proof of Lemma \ref{lem:new_blago}, we have 
\[
\p_t u_i^{\widetilde{f}}=u_i^{\p_t \widetilde{f}}=u_i^f,\ i=1,2.
\]
Hence, an application of Lemma \ref{lm:exact_control} implies the estimate \eqref{eq:lm_exact} for the source $\tilde f$, from which \eqref{eq:exact_no_derivative} follows immediately. This completes the proof of Corollary \ref{cor:smooth_control}.
\end{proof}

\subsection{Proof of Theorem \ref{thm:basic_uniqueness_thm}}
\label{subsec:proof}

For an open set $\mathcal{V}\subset N$ and a  continuous function $h:\overline{\mathcal{V}}\to \R$, we define an ``upside down standing conical set''
\[
\mathcal{B}(\mathcal{V},h;T):=\{(t,y)\in (0,T)\times \mathcal{V}\mid T-h(y)<t\}.
\]

In the next lemma we generalize the classical finite speed of wave propagation and  the approximate controllability results for waves whose interior sources lie in $\mathcal{B}(\mathcal{V},h;T)$.

\begin{lem}
\label{lem:adjusted_finite_speed_of_prop}
Let $T>0$. If $\mathcal{V} \subset N$ is open, $h \colon \overline{\mathcal{V}} \to \R$ is continuous, and 
$f \in C^\infty_0(\mathcal{B}(\mathcal{V},h;T))$, we have
\[
\text{supp}(u^f(T, \cdot))\subset M(\mathcal{V},h),
\]
where the latter set, the modified domain of influence, is defined as in \eqref{eq:adjusted_domain_of_influence}.

Moreover, the set
\[
\{u^f(T,\cdot): \: f \in C^\infty_0(\mathcal{B}(\mathcal{V},h;T))\}
\]
is dense in $L^2(M(\mathcal{V},h))$.
\end{lem}
\begin{proof}
The proof follows from the arguments of \cite[Theorem 2.26]{saksala2025inverse} combined with a partition of unity argument. We shall omit the details.
\end{proof}

In the next two lemmas we investigate the properties of the conjugated operator $\kappa \L \kappa^{-1}$, where $\kappa$ is a smooth non-vanishing complex-valued function.

\begin{lem}
\label{lem:from_equality_of_waves_to_equality_of_LOTs}
Let $y\in \cR$ and $\varepsilon>0$ be such that $B(y,\varepsilon)\subset \cR$. Let $U$ be an open subset of $N$ satisfying $U\cap B(y,\varepsilon)=\emptyset$, and let $r>0$ be such that $U\subset M(B(y,\varepsilon),r)$. If there is a nowhere vanishing function $\kappa\in C^\infty(U)$, for which 
\begin{equation}
\label{eq:condition_lm_4.6}
u_1^\phi(T,x)=\kappa(x)u_2^\phi(T,x),
\quad \text{ for all } \phi\in C_0^\infty(\mathcal{B}(B(y,\varepsilon),r;T)) \text{ and } x\in U,
\end{equation}
we have
\[
\mathcal{L}_{A_1,q_1}\psi=\kappa \mathcal{L}_{A_2,q_2}\kappa^{-1}\psi, 
\quad \text{ for all } \psi \in C^\infty_0(U),
\]
as well as
\begin{equation}
\label{eq:lm_4.6_conclu_2}
A_1=A_2+i\kappa^{-1}d\kappa, 
\quad \text{ and } \quad  
q_1=q_2
\quad 
\text{in }U.
\end{equation}
Moreover, $\kappa|_U$ has a constant magnitude.
\end{lem}
\begin{proof}
Let $\phi\in C_0^\infty(\mathcal{B}(B(y,\varepsilon),r;T))$. We note that there is $\delta>0$ such that $\supp(\phi)\subset [r+\delta,T]\times B(y,\varepsilon)$.
We write $\phi_s(t,\cdot)=\phi(s+t,\cdot)$ for $s\in [0,\delta)$ and notice that $\phi_s\in C_0^\infty(\mathcal{B}(B(y,\varepsilon),r;T))$. By the time transition invariance of the operator $\Box$, we have 
\begin{equation}
\label{eq:lm_4.6.1}
u_k^{\phi_s}(T,\cdot)=u_k^{\phi}(T+s,\cdot) \: \text{ for all } s\in [0,\delta),\ k=1,2.
\end{equation}
Since $\phi_s\in C_0^\infty(\mathcal{B}(B(y,\varepsilon),r;T))$ for all $x\in U$, combining \eqref{eq:condition_lm_4.6} and \eqref{eq:lm_4.6.1} gives us
\[
u_1^{\phi}(T+s,x)=u_1^{\phi_s}(T,x)=\kappa(x)u_2^{\phi_s}(T,x)=\kappa(x)u_2^{\phi}(T+s,x) \text{ for all } s\in [0,\delta), \: x \in U.
\]
By differentiating the equation above twice with respect to $s$ and noticing that   $B(y,\varepsilon)\cap U=\emptyset$
implies $\phi(T,\cdot)\kappa=0$ in $U$, we arrive at 
\[
\mathcal{L}_{A_1,q_1}u_1^\phi(T,x)
=
\kappa(x)\mathcal{L}_{A_2,q_2}u_2^\phi(T,x)
\text{ for all } x\in U.
\]
Thus, a second application of \eqref{eq:condition_lm_4.6} yields that
\[
\mathcal{L}_{A_1,q_1} \kappa u_2^\phi(T,\cdot)=\kappa\mathcal{L}_{A_2,q_2}u_2^\phi(T,\cdot)
\text{ in } U, 
\text{ for all } \phi\in C_0^\infty(\mathcal{B}(B(y,\varepsilon),r;T)). 
\]  

According to Proposition \ref{thm:approximate_controllability} and the Sobolev embedding theorem, the space
\[
\{u_2^\phi(T,\cdot)|_U\mid \phi\in C_0^\infty(\mathcal{B}(B(y,\varepsilon),r;T))\}
\]
is dense in $C^2(U)$ with respect to the $C^2$-norm. Since $\mathcal{L}_{A_1,q_1}$ and $\kappa \mathcal{L}_{A_2,q_2}\kappa^{-1}$ are second order operators, for all $\psi\in C_0^\infty (U)$, we have 
\begin{equation}
\label{eq:lm_4.6.2}
    \mathcal{L}_{A_1,q_1} \psi=\kappa \mathcal{L}_{A_2,q_2}\kappa^{-1} \psi.
\end{equation}
By working in local coordinates and choosing appropriate function $\psi\in C_0^\infty(U)$, we observe that \eqref{eq:lm_4.6_conclu_2} holds pointwise. Furthermore, \eqref{eq:lm_4.6.2} indicates that $\kappa \mathcal{L}_{A_2,q_2}\kappa^{-1}$ is a magnetic Schr\"odinger operator on $C_0^\infty(U)$. Therefore, we get by replicating the argument of \cite[Proposition 2.29]{saksala2025inverse} that $\kappa|_U$ has a constant magnitude. This completes the proof.
\end{proof}

\begin{lem}
\label{lm:kappa_trans_solution}
Let $T>\diam (N,g)$. Suppose that there exists a smooth non-vanishing function $\kappa:N\to\C$ such that
\[
\mathcal{L}_{A_1,q_1}(v) = \kappa \mathcal{L}_{A_2,q_2}(\kappa^{-1} v) \text{ for all }v\in C^\infty(N).
\]
Then we have  
\begin{equation}
\label{eq:func_kappa_trans}
u_1^f= \kappa u_2^{\kappa^{-1} f}
\quad 
\text{ in } (0,2T) \times N,
\end{equation}
whenever $f\in C_0^\infty((0,2T)\times N)$.

Moreover, the equality $\Lambda_{\cS,\cR,2T}^{A_1,q_1}=\Lambda_{\cS,\cR,2T}^{A_2,q_2}$ implies that $\abs{\kappa}=1$ in $N$ and $\kappa|_{\mathcal{R}\cup \mathcal{S}} = 1$.
\end{lem}
\begin{proof}
If $f\in C_0^\infty((0,2T)\times N)$, by hypothesis, we have 
\begin{equation}
\label{eq:conjugated_elliptic_eq}
\cL_{A_1,q_1}u_1^f(t,\cdot)=\kappa\cL_{A_2,q_2}\kappa^{-1}u_1^f(t,\cdot)
\quad \text{ in } N
\end{equation}
for all $t>0$. As $u_1^f$ solves the Cauchy problem \eqref{eq:Cauchy_Problem} with $A=A_1$ and $q=q_1$, and  $\kappa$ is independent of $t$, we get from \eqref{eq:conjugated_elliptic_eq} that  $\kappa^{-1} u_1^f$ solves the following Cauchy problem:
\[
\begin{cases}
(\p_t^2+\cL_{A_2,q_2})v=\kappa^{-1}f 
& \text{ in } (0,2T) \times N,
\\
v(0,\cdot)=\p_tv(0,\cdot)=0 
& \text{ in } N.
\end{cases}
\]
Since $u_2^{\kappa^{-1}f}$ also solves the  problem above, we get from the uniqueness of  solutions that $u_1^f=\kappa u_2^{\kappa^{-1} f}$. This proves the equation \eqref{eq:func_kappa_trans}.

Then we assume that $\Lambda_{\cS,\cR,2T}^{A_1,q_1}=\Lambda_{\cS,\cR,2T}^{A_2,q_2}$. Hence, the equation \eqref{eq:func_kappa_trans} implies that for any $f\in C_0^\infty((0,2T)\times \mathcal{S})$, it holds that
\[
u_2^f=\kappa u_2^{\kappa^{-1} f} 
\quad \text{ in } (0,2T)\times \cR.
\]
Moreover, due to the linearity of the Cauchy problem \eqref{eq:Cauchy_Problem}, we have for a fixed $p \in \cR$  that 
\begin{equation}
\label{eq:multiple_of_kappa_p}
u_2^f=\kappa(p)u_2^{\kappa(p)^{-1}f} \quad \text{ in } (0,2T)\times N.
\end{equation}
By combining the previous two equations, we get for all  $t \in (0,2T)$ that
\[
0=\kappa(p) u_2^{\kappa^{-1} f}(t,p)-\kappa(p)u_2^{\kappa(p)^{-1}f}(t,p)=\kappa(p)u_2^{(\kappa^{-1}-\kappa(p)^{-1})f}(t,p).
\]
Since $\kappa(p)\neq 0$, we obtain
\[
u_2^{(\kappa^{-1}-\kappa(p)^{-1})f}(t,p)=0, \quad  \text{ for all } t \in (0,2T).
\]
In particular, this equation holds for all $ f\in C_0^\infty((0,2T)\times \mathcal{S})$. By Proposition \ref{thm:approximate_controllability}, we must have
$\kappa|_\cS \equiv \kappa(p)$. 

In view of \eqref{eq:func_kappa_trans} and \eqref{eq:multiple_of_kappa_p}, we have for every $ f\in C_0^\infty((0,2T)\times \mathcal{S})$ that
\[
u_1^f(t,x)
=
\kappa(x)u_2^{\kappa f}(t,x)
=
\kappa(p)u_2^{\kappa(p)^{-1}f}(t,x)=u_2^f(t,x),
\quad \text{ when } t\in[0,2T],\ x\in \cS.
\]
Therefore, we obtain  $\Lambda_{\cS,\cS,2T}^{A_1,q_1}=\Lambda_{\cS,\cS,2T}^{A_2,q_2}$, and an application of \cite[Theorem 1.1]{saksala2025inverse} gives us $\abs{\kappa}=1$ in $N$ and $\kappa(p)=\kappa|_\cS=1$. Since $p\in \cR$ was arbitrarily chosen, it follows immediately that $\kappa|_\cR=1$. This completes the proof of Lemma \ref{lm:kappa_trans_solution}.
\end{proof}

In the next lemma, we introduce an important convergence result, which will be one of the key elements in the proof of Theorem \ref{thm:basic_uniqueness_thm}. An analogous result was given in \cite[Proposition 2]{kian2019unique} for the solutions of the hyperbolic initial boundary value problem \eqref{eq:IBVP}.

\begin{lem}
\label{lm:convergence}
Let $x_0 \in N$ be an arbitrary point, and let $B\subset N$ be its neighborhood. Let $\tilde X_{k} \subset N$ be a sequence of sets such that each contains a neighborhood of $x_0$ and $\lim_{k\to \infty}\diam (\tilde X_{k})=0$. Let $\psi_0\in C_0^\infty(B)$ be such that $\psi_0(x_0)\neq 0$. 

If $\{f_{kl}\}_{k,l=1}^\infty$ is a sequence of functions in $L^2((0,T)\times \cS)$ that satisfies the following properties:
\begin{enumerate}
\item for all $k\geq 1$, there exists a function $v_k\in L^2(N)$ such that the sequence $\{u^{f_{kl}}(T,\cdot)\}_{l=1}^\infty$ converges weakly to $v_k$ in $L^2(N)$,
\item there exists a constant $C>0$ such that $\norm{v_k}_{L^2(N)}\leq C\text{Vol}(\tilde X_{k})^{-1/2}$ for $k\geq 1$.
\item $\supp(v_k)\subset \overline{\tilde X_{k}}\cup (N\setminus B)$ for $k\geq 1$,
\item $\{\inner{v_k,\psi_0}\}_{k=1}^\infty$ converges,
\end{enumerate}
there exists $\kappa\in \mathbb{C}$ such that 
\[
\lim_{k\to\infty}\lim_{l\to\infty}\inner{u^{f_{kl}}(T,\cdot),\psi}=\kappa \psi(x_0) \text{ for all } \psi\in C_0^\infty(B).
\]
\end{lem}
\begin{proof}
Let $\psi\in C_0^\infty(B)$. Firstly, we get from the assumption  (3) that $\supp(v_k\psi)\subset \overline{\tilde X_{k}}$. In what follows we adopt the notation $1_{\tilde X_{k}}$ for the characteristic function of the set $\tilde X_{k}$ and  further write 
\begin{equation}
\label{eq:v_k_and_psi_inner_prod}
\inner{v_k,\psi}=\psi(x_0)\inner{v_k,1_{\tilde X_{k}}}+R_{k,x_0}(\psi),
\end{equation}
where
\[
R_{k,x_0}(\psi):=\int_{\tilde X_{k}}v_k(y)(\psi(y)-\psi(x_0))d V_g(y).
\]

Since $\lim_{k\to \infty}\diam (\tilde X_{k})=0$ and $x_0\in \tilde X_{k}$ for all $k\in \N$,  there exists a coordinate neighborhood $U$ of $x_0$ such that $\tilde X_{k}\subset U$ for all large enough $k$. Moreover, we can choose this neighborhood in such a way that the Riemannian distance and the Euclidean distance are bi-Lipschitz equivalent in $U$.
and the Cauchy-Schwarz inequality, we have
\begin{equation}
\label{eq:conv_of_remainder}
\begin{aligned}
\abs{R_{k,x_0}(\psi)}&\leq \int_{\tilde X_{k}}\abs{v_k(y)}\abs{\psi(x_0)-\psi(y)}d V_g(y)
\\
&\leq C\norm{d\psi}_{L^\infty(N)}\int_{\tilde X_{k}}\abs{v_k(y)}d(x_0,y)d V_g(y)
\\
&\leq  C\norm{d\psi}_{L^\infty(N)} \norm{v_k}_{L^2(N)}\diam(\tilde X_{k})\text{Vol}(\tilde X_{k})^{1/2},
\end{aligned}
\end{equation}
where $dV_g$ is the Riemannian volume, and the constant $C$ depends only on $(N,g)$. Thus, the assumption (2) yields that $R_{k,x_0}(\psi)\to 0$ as $k\to\infty$. 

When we set $\psi=\psi_0$, we get from \eqref{eq:v_k_and_psi_inner_prod}, \eqref{eq:conv_of_remainder}, and the assumption (4) that $\inner{v_k,1_{\tilde X_{k}}}$ converges to some $\kappa \in \C$. Therefore, we obtain from the assumption (1), formula \eqref{eq:v_k_and_psi_inner_prod}, and $R_{k,x}(\psi)\to 0$ as $k\to\infty$ that
\[
\lim_{k\to\infty}\lim_{l\to\infty}\inner{u^{f_{kl}}(T,\cdot),\psi}=\lim_{k\to\infty}\inner{v_k,\psi}=\kappa\psi(x_0),
\quad \text{ for all }  \psi\in C_0^\infty(B).
\]
This completes the proof of Lemma \ref{lm:convergence}.
\end{proof}

In our final technical lemma, we first show that if we choose the sets $(\tilde X_k)_{k=1}^\infty$ and the sources $\{f_{kl}\}_{k,l=1}^\infty$ $L^2((0,T)\times \cS)$ appropriately,  the waves $u_i^{f_{kl}}$ for both $i \in \{1,2\}$ satisfy the conditions of Lemma \ref{lm:convergence}. This allows us to prove that the values $u_1^\phi(T,x_0)$ and $u_2^\phi(T,x_0)$ agree up to a multiplicative constant that does not depend on the source $\phi$ whenever $\phi$ is a smooth and compactly supported function in a certain open subset of $(0,T)\times \cR$. 

\begin{lem}
\label{lem:u_1_=_kappa_u_2_at_x_0}
Let $x_0,y,s, \gamma$, and $\varepsilon$ be as before Proposition \ref{prop:Z-set}, and let $\delta\in (0,\varepsilon]$. Define
\begin{align*}
y_{1}:=\gamma(\varepsilon-\delta),
\quad 
r_{1}:=s-\varepsilon+2\delta,
\quad 
B_1:=B(y_1,r_1),
\\
y_{2}:=\gamma(-\varepsilon-\delta)
\quad
r_{2}:=s+\varepsilon-\delta,
\quad
B_2:=B(y_2,r_2)
\end{align*}

Let $T>\diam (N,g)$. If $\rho \in (0,r_1)$ is such that $B(y_1,\rho)\subset \cR$,  there exists a non-zero complex number $\kappa(x_0)$ for which
\begin{equation}
\label{eq:u_1_=_kappa_u_2_at_x_0}
\kappa(x_0) u_2^\phi(T,x_0)=u_1^\phi(T,x_0),
\quad 
\text{ for all } \phi\in C_0^\infty(\mathcal{B}(B(y_1,\rho), r_1-\rho;T)).
\end{equation}
\end{lem}

\begin{proof}
According to \eqref{eq:def_Z_delta},  we have  $Z_\delta=\overline{B_1\setminus B_2}$. When $\delta>0$ is small enough, it follows from Assumption \eqref{eq:assumption_rev} and Proposition \ref{prop:Z-set} that the set $Z_\delta$ contains an open neighborhood of $x_0$ satisfying the condition \eqref{con:convex}. Finally, for each $k \in \N$ we define the set $X_{Z_\delta^\circ, k,x_0}:=Z_\delta^\circ\setminus M(\cR,h_{x_0,k})$ as originally stated in the equation \eqref{eq:set_X_{k,x}}. Since $Z_\delta$ and $x_0$ stay the same throughout the proof, we shall use the shorthand notation $X_{Z_\delta^\circ, k,x_0}=X_k$ within this proof.

Due to Proposition \ref{thm:approximate_controllability}, for each $k \in \N$ we can choose the sources $\{f_{kl}\}_{l=1}^\infty$ in $L^2((0,T)\times\cS)$ such that the sequence $u_1^{f_{kl}}(T,\cdot)$ converges to $\frac{1_{X_{k}}}{\text{Vol}(X_{k})}$ in $L^2(N)$ as $l$ increases. We aim to show first that the sequences $u_i^{f_{kl}}(T,\cdot)$ for both $i \in \{1,2\}$ satisfy the conditions (1)--(4) of Lemma \ref{lm:convergence} for $B=B_1$.

As strong convergence implies  weak convergence,  the waves $u_1^{f_{kl}}(T,\cdot)$ satisfy Condition (1) if we denote the limit function $\frac{1_{X_{k}}}{\text{Vol}(X_{k})}$ by $v_k$.  
Conditions (2) and (3) then follow immediately. Moreover, due to the Lebesgue differentiation theorem, Condition (4) holds for any $\psi_0\in C_0^\infty(B_1)$ such that $\psi_0(x_0)\neq 0$. 

Let us next show that the  sequence $\{u_2^{f_{kl}}(T,\cdot)\}_{k,l=1}^\infty$ also satisfies the conditions (1)--(4) of Lemma \ref{lm:convergence}, although with possibly different choices for $(v_k)_{k=1}^\infty$ and $\psi_0$ than those we made for the sequence $u_1^{f_{kl}}(T,\cdot)$.

Condition (1): 
Since the sources $\{f_{kl}\}_{k,l=1}^\infty$ are chosen from $L^2((0,T)\times\cS)$, it follows from 
Lemma \ref{lem:new_blago}
that for all $\phi\in L^2((0,T)\times\cR)$ and $k \in \N$, we have
\begin{equation}
\label{eq:conv_and_blago}
\begin{aligned}
\lim_{l\to\infty}\inner{u_2^{f_{kl}}(T,\cdot),\p_t u_2^\phi(T,\cdot)}
=&
\lim_{l\to\infty}\inner{u_1^{f_{kl}}(T,\cdot),\p_t u_1^\phi(T,\cdot)}
=
\inner{\frac{1_{X_{k}}}{\text{Vol}(X_{k})},\p_t u_1^\phi(T,\cdot)}.
\end{aligned}
\end{equation}

We seek to define a bounded linear functional 
\begin{align*}
G_{k}\colon L^2(N)\to \R,
\quad 
G_{k}(w)= \inner{&\frac{1_{X_{k}}}{\text{Vol}(X_{k})},\p_t u_1^{\phi_w}(T,\cdot)},
\end{align*}
where $\phi_w\in L^2((0,T)\times\cR)$ is such that $\p_t u_2^{\phi_w}(T,\cdot)=w$, 
and the function $u_i^{\phi_w}$ solves the Cauchy problem \eqref{eq:Cauchy_Problem} with the interior source $f=\phi_w$ and the coefficients $A=A_i$ and $q=q_i$. 
Let us first show that the map $G_{k}$ is well-defined. 
To this end, due to Proposition \ref{prop:L2_control}, for each $w \in L^2(N)$ we can  choose a source function $\phi_w \in L^2((0,T)\times\cR)$ such that $\p_t u_2^{\phi_w}(T,\cdot)=w$. 
Let $\tilde \phi_w \in L^2((0,T)\times\cR)$ be another such source.  Then we get from the linearity of the wave equation that 
\[
0=\p_t u_2^{\phi_w}(T,\cdot)-\p_t u_2^{\tilde \phi_w}(T,\cdot)
=
\p_t u_2^{\phi_w-\tilde \phi_w}(T,\cdot).
\]
Since $\phi_w-\tilde \phi_w \in L^2((0,T)\times \cR)$, an application of the equation \eqref{eq:conv_and_blago} for $\phi:=\phi_w-\tilde \phi_w$ yields that
\[
\inner{\frac{1_{X_{k}}}{\text{Vol}(X_{k})},\p_t u_1^{\phi_w}(T,\cdot)}
-
\inner{\frac{1_{X_{k}}}{\text{Vol}(X_{k})},\p_t u_1^{\tilde \phi_w}(T,\cdot)}
=
0.
\]
Hence,   the map $G_{k}$ is well-defined.
Clearly, the linearity of $G_{k}$ follows from the linearity of the Cauchy problem \eqref{eq:Cauchy_Problem}.
Thus, it only remains to verify that the map $G_{k}$ is bounded. 
Let $w \in L^2(N)$. By the Cauchy-Schwarz inequality and Lemma \ref{lm:exact_control}, we obtain the estimate
\[
|G_{k}(w)|
\leq 
C\|\p_t u_1^{\phi_w}(T,\cdot)\|_{L^2(N)}
\leq
C\|\p_t u_2^{\phi_w}(T,\cdot)\|_{L^2(N)} 
=
C\|w\|_{L^2(N)}. 
\]

Since the operator $G_{k}$ is bounded, we   deduce from the Riesz representation theorem that there exists a function $v_k\in L^2(N)$ such that $G_{k}(w)=\inner{v_k,w}$.
Thus, the equation \eqref{eq:conv_and_blago} implies that for every $k \in \N$ the sequence $\{u_2^{f_{kl}}(T,\cdot)\}_{l=1}^\infty$ is weakly $L^2$-convergent with the limit $v_k$. Moreover, due to \eqref{eq:conv_and_blago}, the function $v_k$  satisfies the identity 
\begin{equation}
\label{eq:pairing_of_v_k_and_wave}
\langle v_k, \p_t u_2^\phi(T,\cdot)\rangle = \Big\langle \frac{1_{X_k}}{\text{Vol}(X_{k})}, \p_t u_1^\phi(T,\cdot)\Big\rangle
\quad
\text{for all } \phi \in L^2((0,T)\times\cR).
\end{equation}
Thus, Condition (1) has been verified. 

Condition (2): 
Let $k \in \N$. By Proposition \ref{prop:L2_control}, there exists a function $\phi_k\in L^2((0,T)\times\cR)$ such that $\p_t u_2^{\phi_k}(T,\cdot)=\text{Vol}(X_{k})^{1/2}v_k$.
By \eqref{eq:pairing_of_v_k_and_wave}, the Cauchy-Schwarz inequality, as well as Lemma \ref{lm:exact_control}, we obtain the estimate
\begin{align*}
\norm{\text{Vol}(X_{k})^{1/2}v_k}_{L^2(N)}^2&=\left|\inner{\text{Vol}(X_{k})^{1/2}v_k,\p_t u_2^{\phi_k}(T,\cdot)}\right|=\left|\inner{\frac{1_{X_{k}}}{\text{Vol}(X_{k})^{1/2}},\p_t u_1^{\phi_k}(T,\cdot)}\right|
\\
&\leq 
\norm{\p_t u_1^{\phi_k}(T,\cdot)}_{L^2(N)}
\leq 
C \norm{\p_t u_2^{\phi_k}(T,\cdot)}_{L^2(N)}=C \norm{\text{Vol}(X_{k})^{1/2}v_k}_{L^2(N)}.
\end{align*}
Here the constant $C$ is the same as in Lemma \ref{lm:exact_control} and is independent of $k$. Therefore, we have
\[
\norm{\text{Vol}(X_{k})^{1/2}v_k}_{L^2(N)} \leq C \text{ for all } k \in \N.
\]
This completes the verification of Condition (2).

\begin{rem}
We would like to emphasize   that if the function $\phi$ as in the equations \eqref{eq:conv_and_blago} and  \eqref{eq:pairing_of_v_k_and_wave} is smooth and compactly supported,  we can replace $\p_t u_i^{\phi}$ with $u_i^\phi$ in these equations via a similar trick as we did in the proof of Corollary \ref{cor:smooth_control}. We will utilize this observation in the verification of Conditions (3) and (4).
\end{rem}

Condition (3): 
Let us recall that
\[
B_i:=B(y_i,r_i),
\quad 
Z_\delta:=\overline{B_1\setminus B_2}, 
\quad \text{ and } \quad 
X_{k}=Z_\delta^\circ \setminus M(\mathcal{\cR},h_{k,x_0}),
\]
where $h_{k,x_0}\colon \cR \to \R$ was defined as $h_k(z)=d(z,x_0)-\frac{1}{k}$. By Lemma \ref{lem:adjusted_finite_speed_of_prop}, we have for any $\phi\in C_0^\infty(\mathcal{B}(\cR,h_{k,x_0};T))$ that  $\supp(u_1^\phi(T,\cdot))\subset M(\cR, h_{k,x_0})$. 
Then we get from the equation \eqref{eq:conv_and_blago} that
\[
\inner{v_k,u_2^\phi(T,\cdot)}=\lim_{l\to\infty}\inner{u_1^{f_{kl}}(T,\cdot),u_1^\phi(T,\cdot)}=\inner{\frac{1_{X_k}}{\text{Vol}(X_k)},u_1^\phi(T,\cdot)}=0.
\]
Hence, it follows from Lemma \ref{lem:adjusted_finite_speed_of_prop} that $v_k$ vanishes in $M(\cR,h_{k,x_0})$. 

Since $y_2 \in \cR$, we can choose $\eta \in (0,r_2)$ such that $B(y_2,\eta)\subset \cR$. By Lemma \ref{lem:adjusted_finite_speed_of_prop}, the set
\begin{equation}
\label{eq:set_of_waves}
\{u_i^\phi(T,\cdot)\mid \phi\in C_0^\infty(\mathcal{B}(B(y_2,\eta), r_2-\eta; T))\}, \quad \text{ for } i \in \{1,2\},
\end{equation}
is dense in $L^2(B_2)$ as
$M(B(y_2,\eta),r_2-\eta)=\overline{B_2}$.
Since $u_1^{f_{kl}}(T,\cdot)\to \frac{1_{X_{k}}}{\text{Vol}(X_{k})}$ in $L^2(N)$ as $l\to\infty$, we get from the equation \eqref{eq:conv_and_blago} that
\[
\inner{v_k,u_2^\phi(T,\cdot)}=\lim_{l\to\infty}\inner{u_1^{f_{kl}}(T,\cdot),u_1^\phi(T,\cdot)}=\inner{\frac{1_{X_{k}}}{\text{Vol}(X_k)},u_1^\phi(T,\cdot)}=0
\]
for all $\phi$ as in \eqref{eq:set_of_waves}. In particular, the last equality above  holds since the sets $X_k$ and $B_2$ are disjoint while supp$(u_1^\phi(T,\cdot))\subset B_2$. Hence, $v_k$ vanishes in $B_2$. 

If $p \in\supp(v_k)$, it follows that either $p \in B_1$ or $p\in N\setminus B_1$. In the first case,   since $B_1\subset Z_\delta\cup B_2$, and we have already excluded the possibility that $p \in B_2$, we must have   $p \in Z_\delta \subset M(\cR, h_{k,x_0})\cup \overline{X_k}$. Furthermore, we have also shown above that $p \notin M(\cR, h_{k,x_0})$. Since $M(\cR, h_{k,x_0})$ is a closure of an open set, it must hold that $p \in \overline{X_k}$. Hence, we  conclude that $\supp(v_k)\subset \overline{X_k}\cup(N\setminus B_1)$. The verifies Condition (3).

Condition (4): Let $\rho \in (0,r_1)$ be such that $B(y_1,\rho)\subset \cR$. Then we have
\[
\mathcal{B}(B(y_1,\rho), r_1-\rho;T)\subset (0,T)\times\cR
\quad \text{ and } \quad
M(B(y_1,\rho),r_1-\rho)=\overline{B_1}.
\]
Since the open ball $B_1$ is an open neighborhood of $x_0$, it follows from Proposition \ref{thm:approximate_controllability} that we can choose $\phi_0\in C_0^\infty(\mathcal{B}(B(y_1,\rho), r_1-\rho;T))$ such that $u_2^{\phi_0}(T,x_0)\neq 0$ and supp$(u^{\phi_0}_2(T,\cdot)) \subset B_1$.  
Hence, we can choose $u_2^{\phi_0}(T,\cdot)$ for the function $\psi_0$ as in Condition (4) of Lemma \ref{lm:convergence}. 
Thus, the equations \eqref{eq:conv_and_blago} and \eqref{eq:pairing_of_v_k_and_wave}, in conjunction with the Lebesgue differentiation theorem, yield that
\[
\begin{aligned}
\lim_{k\to\infty}\inner{v_k,\psi_0}
=&
\lim_{k\to\infty}\inner{v_k,u_2^{\phi_0}(T,\cdot)}
=
\lim_{k\to\infty}\inner{\frac{1_{X_{k}}}{\text{Vol}(X_k)},u_1^{\phi_{0}}(T,\cdot)}=u_1^{\phi_0}(T,x_0).
\end{aligned}
\]
Consequently, Condition (4) of Lemma \ref{lm:convergence} is fulfilled.
Therefore, we have completed showing that $u_2^{f_{kl}}(T,\cdot)$ satisfies Conditions (1)--(4) in Lemma \ref{lm:convergence} for $B=B_1$ and $\psi_0=u^{\phi_0}_2(T,\cdot)$.

Since supp $(u^\phi_2(T,\cdot)) \subset B_1$ for all $\phi\in C_0^\infty(\mathcal{B}(B(y_1,\rho), r_1-\rho;T))$, Lemma \ref{lm:convergence}  implies that
\[
\lim_{k\to\infty}\lim_{l\to\infty}\inner{u_2^{f_{kl}}(T,\cdot),u^\phi_2(T,\cdot)}=\kappa(x_0) u_2^\phi(T,x_0).
\]
On the other hand, we obtain from  Lemma \ref{lem:Blago_identity} and the Lebesgue differentiation theorem that
\[
\lim_{k\to\infty}\lim_{l\to\infty}\inner{u_2^{f_{kl}}(T,\cdot),u^\phi_2(T)}
=
\lim_{k\to\infty}\lim_{l\to\infty}\inner{u_1^{f_{kl}}(T,\cdot),u^\phi_1(T,\cdot)}
=
u^\phi_1(T,x_0).
\]
Therefore, we arrive at the equation \eqref{eq:u_1_=_kappa_u_2_at_x_0} by combining these two observations. This completes the proof of Lemma \ref{lem:u_1_=_kappa_u_2_at_x_0}.
\end{proof}

We are now ready to present the proof of our main theorem.

\begin{proof}[Proof of Theorem \ref{thm:basic_uniqueness_thm}]
We adopt the notations of Lemma \ref{lem:u_1_=_kappa_u_2_at_x_0} and first show that the equation \eqref{eq:u_1_=_kappa_u_2_at_x_0} holds in the set $Z^\circ_\delta$ for some smooth and non-vanishing complex-valued function $\kappa$. As earlier, here we write  $Z_\delta^\circ$  for the interior of the set $Z_\delta$. According to Proposition  \ref{prop:Z-set}, $Z^\circ_\delta$  is nonempty, and due to the choice of $\delta$, it satisfies the property \eqref{con:convex}. Furthermore,  since
\[
Z_\delta=\overline{B_1\setminus B_2}, 
\quad \text{ and } \quad
X_{Z_\delta^\circ,k,x}=Z_\delta^\circ \setminus M(\cR, h_{k,x}), 
\]
we have $X_{Z_\delta^\circ,x,k} \subset B_1$ for each $x\in Z_\delta^\circ$ and $k \in \N$.

Then we choose a sequence of sources $\{f_{kl}\}_{k,l=1}^\infty$ from $L^2((0,T)\times\cS)$ such that the respective waves $u_1^{f_{kl}}(T,\cdot)$ converge to $\frac{1_{X_{Z_\delta^\circ,k,y}}}{\text{Vol}(X_{Z_\delta^\circ,k,y})}$ in $L^2(N)$ as $l\to \infty$. As in the previous parts of this proof, the sequences $u_i^{f_{kl}}(T,\cdot)$ for $i\in \{1,2\}$ satisfy the conditions (1)--(4) of Lemma \ref{lm:convergence}. Thus, we obtain the equation 
\begin{equation}
\label{eq:u_1_=_kappa_u_2_local}
\kappa(x) u_2^\phi(T,x)=u_1^\phi(T,x)
\quad 
\text{ for all }
\phi\in C_0^\infty(\mathcal{B}(B(y_1,\rho), r_1-\rho;T)). 
\end{equation}
In particular, we have shown the existence of a function $\kappa \colon Z_\delta^\circ \to \C$ such that \eqref{eq:u_1_=_kappa_u_2_local} holds for all $x \in Z^\circ_\delta$.

Let us now proceed to show that the function $\kappa$ is smooth and non-vanishing. By Proposition \ref{thm:approximate_controllability}, for any $x\in Z_\delta^\circ$ and $k \in \{1,2\}$, we can find 
$\phi_k\in C_0^\infty(\mathcal{B}(B(y_1,\rho), r_1-\rho;T))$ and a neighborhood $U_x$ of $x$ such that $u_k^{\phi_k}(T,\cdot)\neq 0$ in $U_x$. Thus $\kappa = \frac{u_2^{\phi_1}(T)}{u_1^{\phi_1}(T)}$ and
$\frac{1}{\kappa}=\frac{u_1^{\phi_2}(T)}{u_2^{\phi_2}(T)}$ are well-defined and smooth in $U_x$. Hence, $\kappa$ does not vanish in  $U_x$. Furthermore, since the set $U_x\subset Z_\delta^\circ$ can be chosen for any $x\in Z_\delta^\circ$, the function $\kappa$ does not vanish and is smooth in $Z_\delta^\circ$. 

To finish the proof, we now show that the function $\kappa$ above can be defined globally. Since the point $x_0 \in N$ was arbitrarily chosen, it follows from the construction above that for any $x\in N$, we can find an open neighborhood $Z_{\delta,x}^\circ$, as in \eqref{eq:def_Z_delta}, 
and a smooth non-vanishing function $\kappa \colon Z_{\delta,x}^\circ \to \C$ that satisfies an equation similar to \eqref{eq:u_1_=_kappa_u_2_local}. Since the sets $\{Z_{\delta,x}^\circ\}_{x\in N}$ form an open cover of the compact manifold $N$, we can choose a finite subcover $\{Z_{i}^\circ\}_{i=1}^I$ of $N$, with the respective functions $\kappa_i$ defined on each $Z_i^\circ$. 

Let us  define
\[
\kappa\colon N \to \C, \quad \kappa(x):=\kappa_i(x), \quad \text{ whenever } x\in Z_i^\circ, \quad \text{ for some } i \in \{1,\ldots,I\},
\]
and show that this function is well-defined. That is,  $x\in Z_i^\circ\cap Z_j^\circ$ implies that $\kappa_i(x)=\kappa_j(x)$. 
Recall that we have defined $Z_i=\overline{B(y_{1,i},r_{1,i})\setminus B(y_{2,i},r_{2,i})}$ for some $y_{j,i}\in N$ and $r_{j,i}>0$. 
Without loss of generality, we only consider the case that $x\in Z_a^\circ\cap Z_b^\circ$ for some $a,b \in \{1,\ldots,I\}$.
Let $\varepsilon>0$ be such that $B(y_{2,a},\varepsilon),B(y_{2,b},\varepsilon)\subset \cR$.
We then choose a function $\psi\in C_0^\infty(Z_a^\circ\cap Z_b^\circ)$ that does not vanish near $x$. We also choose the sequences $\{\phi_i\}_{i=1}^\infty\subset C_0^\infty(\mathcal{B}(B(y_{2,a},\varepsilon),r_{2,a}-\varepsilon;T))$ and $\{\varphi_i\}_{i=1}^\infty\subset C_0^\infty(\mathcal{B}(B(y_{2,b},\varepsilon),r_{2,b}-\varepsilon;T))$
such that $u_1^{\phi_j}(T,\cdot)\to \psi$ and $u_1^{\varphi_j}(T,\cdot)\to \psi$ in $L^2(N)$.
Since  we have chosen $\phi_j,\varphi_j\in C_0^\infty((0,2T)\times \cR)$ for each $j \in \N$, and   the sequences 
\[
\left\{\norm{u_1^{\phi_j}(T,\cdot)}_{L^2(N)}\right\}_{j=1}^\infty 
\quad \text{and} \quad  
\left\{\norm{u_1^{\varphi_j}(T,\cdot)}_{L^2(N)}\right\}_{j=1}^\infty
\]
are bounded, by Corollary \ref{cor:smooth_control}, the sequences 
\[
\left\{\norm{u_2^{\phi_j}(T,\cdot)}_{L^2(N)}\right\}_{j=1}^\infty
\quad \text{and} \quad
\left\{\norm{u_2^{\varphi_j}(T,\cdot)}_{L^2(N)}\right\}_{j=1}^\infty
\]
are also bounded. 

Let us next argue why the sequences $\{u_2^{\phi_j}(T,\cdot)\}_{j=1}^\infty$ and $\{u_2^{\varphi_j}(T,\cdot)\}$ are weakly convergent. To this end, we take any function $v\in L^2(N)$ and use the approximate controllability, Lemma \ref{lem:adjusted_finite_speed_of_prop} to choose a sequence $\{h_{i}\}_{i=1}^\infty$ in $C_0^\infty((0,\infty)\times \cS)$ such that $u_2^{h_{i}}(T,\cdot)\to v$. We show first that $\{u_1^{h_i}(T,\cdot)\}_{i=1}^\infty$ is weakly convergent.
Indeed, due to the exact controllability of the Cauchy Problem \eqref{eq:Cauchy_Problem}  from the set $\cR$ at time $T$, we have
\[
\{\p_t u_1^g(T,\cdot)\mid g\in L^2((0,\infty)\times\cR)\}=L^2(N).
\]
Hence, we choose $g\in L^2((0,\infty)\times\cR)$ such that $\p_t u_1^g(T,\cdot)=v$ and apply Lemma \ref{lem:new_blago} to see that
\begin{align*}
\lim_{j\to\infty} \inner{u_1^{h_j}(T,\cdot),v}
=
\lim_{j\to\infty} \inner{u_1^{h_j}(T,\cdot),\p_t u_1^g(T,\cdot)}
=
\lim_{j\to\infty}\inner{u_2^{h_j}(T,\cdot),\p_t u_2^g(T,\cdot)}
=
\inner{v,\p_t u_2^g(T,\cdot)}.
\end{align*}
Therefore,  $\{u_1^{h_j}(T,\cdot)\}_{j=1}^\infty$ is weakly convergent.

Applying Lemma \ref{lem:new_blago} again, we have for all $j \in \N$ that
\begin{equation}
\label{eq:weak_conv_of_u_2_justification}
\begin{aligned}
\inner{u_2^{\phi_j}(T,\cdot),v}
&=
\inner{u_2^{\phi_j}(T,\cdot),v-u_2^{h_j}(T,\cdot)}+\inner{u_2^{\phi_j}(T,\cdot),u_2^{h_j}(T,\cdot)}
\\
&=
\inner{u_2^{\phi_j}(T,\cdot),v-u_2^{h_j}(T,\cdot)}+\inner{u_1^{\phi_j}(T,\cdot),u_1^{h_j}(T,\cdot)}
\\
&=
\inner{u_2^{\phi_j}(T,\cdot),v-u_2^{h_j}(T,\cdot)}
+\inner{u_1^{\phi_j}(T,\cdot)-\psi,u_1^{h_j}(T,\cdot)}
+\inner{\psi,u_1^{h_j}(T,\cdot)}.
\end{aligned}    
\end{equation}
Since the sequence $\left\{\norm{u_2^{\phi_j}(T,\cdot)}_{L^2(N)}\right\}_{j=1}^\infty$ is bounded and $v=\lim_{j\to\infty}u_2^{h_j}(T,\cdot)$,   the first term on the right-hand side of \eqref{eq:weak_conv_of_u_2_justification} vanishes at the limit $j \to \infty$. On the other hand, as $\{u_1^{h_j}(T,\cdot)\}_{j=1}^\infty$ is weakly convergent and therefore bounded, and $\psi=\lim_{j\to \infty } u_1^{\phi_j}(T,\cdot)$,    the second term on the right-hand side of \eqref{eq:weak_conv_of_u_2_justification} also vanishes as $j \to \infty$, while the third term converges in $\R$. Thus, $\{u_2^{\phi_j}(T,\cdot)\}_{j=1}^\infty$ is weakly convergent.
By analogous arguments, we conclude that $\{u_2^{\varphi_j}(T,\cdot)\}_{j=1}^\infty$ is also weakly convergent. 
In what follows we shall denote the weak limits of $\{u_2^{\phi_j}(T,\cdot)\}_{j=1}^\infty$ and $\{u_2^{\varphi_j}(T,\cdot)\}_{j=1}^\infty$ by $\xi$ and $\zeta$, respectively.
Due to the properties of the functions $\kappa_a$ and $\kappa_b$, we have for each $i \in \N$ that
\[
u_1^{\phi_i}(T,\cdot)=\kappa_au_2^{\phi_i}(T,\cdot) \text{ in }Z_a^\circ,
\quad \text{ and } \quad
u_1^{\varphi_i}(T,\cdot)=\kappa_bu_2^{\varphi_i}(T,\cdot) \text{ in }Z_b^\circ.
\]
Since $\psi=\lim_{i\to \infty}u^{\phi_i}_1(T,\cdot)|_{Z_a}=\lim_{i\to \infty}u^{\varphi_i}_1(T,\cdot)|_{Z_b}$, we obtain from these equations  that $\lim_{i\to\infty}\kappa_au_2^{\phi_i}(T,\cdot)|_{Z_a}=\psi$ and $\lim_{i\to\infty}\kappa_bu_2^{\varphi_i}(T,\cdot)|_{Z_b}=\psi$ in the $L^2$-sense.
Let us recall that $u_2^{\phi_i}(T,\cdot)|_{Z_a^\circ}$ and $u_2^{\varphi_i}(T,\cdot)|_{Z_b^\circ}$  converge weakly to $\xi|_{Z_a^\circ}$ and $\zeta|_{Z_b^\circ}$, respectively.
Hence, we get that

\[
\psi=\kappa_a\xi  \text{ in } Z_a^\circ,
\quad \text{ and } \quad
\psi=\kappa_b\zeta  \text{ in } Z_b^\circ.
\]
Since $\phi_i-\varphi_i \in C^\infty_0((0,T),\cR)$ and $u_1^{\phi_i-\varphi_i}(T,\cdot)\to 0$ as $i \to \infty$,  we deduce from Lemma \ref{lem:new_blago} that
\[
\inner{\xi-\zeta,u_2^h(T,\cdot)}
=
\lim_{i \to \infty}\inner{u_2^{\phi_i-\varphi_i}(T,\cdot),u_2^h(T,\cdot)}
=
\lim_{i \to \infty}\inner{u_1^{\phi_i-\varphi_i}(T,\cdot),u_1^h(T,\cdot)}
=
0
\]
for all $h \in C^\infty_0((0,T)\times \cS)$. Thus, an application of Lemma \ref{lem:adjusted_finite_speed_of_prop}  yields that 
\[
0=\xi-\zeta=\left(\frac{1}{\kappa_a}-\frac{1}{\kappa_b}\right)\psi
\quad 
\text{ in } Z_a^\circ \cap Z_b^\circ.
\]
As $\psi$ does not vanish at the point $x$, we see that $\kappa_a=\kappa_b$ near $x$.

Thanks to Lemma \ref{lem:from_equality_of_waves_to_equality_of_LOTs}, we have $A_1=A_2+i\kappa_j^{-1}d\kappa$ and $q_1=q_2$ in $Z^\circ_j$, and there is a constant $C_j$ such that $\abs{\kappa_j}=C_j$ for each $j\in \{1,\ldots,I\}$. Since $\kappa|_{Z^\circ_j}=\kappa_j$ and the sets $\{Z^\circ_1,\ldots, Z^\circ_I\}$ cover $N$, we conclude that
\begin{equation}
\label{eq:equality_of_coefficients}
A_1=A_2+i\kappa^{-1}d\kappa 
\quad 
\text{ and }
\quad
q_1=q_2 \quad
\text{ in } N.
\end{equation}

To finish the proof, we still need to show that $\kappa$ is a unitary function that equals to one in $\cS \cup \cR$. Let $2\varrho>0$ be a Lebesgue number for the cover $\{Z^\circ_1,\ldots, Z^\circ_I\}$, see for instance \cite[Lemma 27.5]{Munkres}. For any points $x,y\in N$, we can find a normalized smooth curve $\gamma:[0,S]\to N$ with $S\geq d(x,y)$ such that $\gamma(0)=x$ and $\gamma(S)=y$. Then for all $t\in [0,S-\varrho]$, we have 
$d(\gamma(t), \gamma(t+\varrho))\leq \varrho$, and due to the choice of Lebesgue number $2\varrho$, there exists $j \in \{1,\ldots,I\}$ such that $\gamma(t), \gamma(t+\varrho) \in Z_j$. In particular, by the previous part of the proof, we know that $\kappa_j$ has a constant magnitude. Therefore,
$
\abs{\kappa(\gamma(t+\varrho))}=\abs{\kappa(\gamma(t))},
$
and an iteration of this argument gives us
\[  \abs{\kappa(x)}=\abs{\kappa(\gamma(0))}=\abs{\kappa(\gamma(S))}=\abs{\kappa(y)}.
\]
Hence, the function $\kappa$ has a constant magnitude in $N$. According to \cite[Proposition 2.39]{saksala2025inverse}, we have
\begin{align*}
\cL_{A_1,q_1}v
&=
\kappa\cL_{A_2,q_2}(\kappa^{-1} v)
\quad \text{ for all } v \in C^\infty(N).
\end{align*}
Therefore, Lemma \ref{lm:kappa_trans_solution} implies that  $\abs{\kappa}=1$ in $N$ and $\kappa|_{\cS\cup \cR}=1$. The proof of Theorem \ref{thm:basic_uniqueness_thm} is now complete.
\end{proof}
\bibliographystyle{abbrv}
\bibliography{bib_source_to_solution}

\end{document}